\documentclass[11pt,english]{amsart}
\usepackage[T1]{fontenc}
\usepackage[latin9]{inputenc}
\usepackage{amsthm}
\usepackage{amstext}
\usepackage{stmaryrd}
\usepackage{cite}
\usepackage{esint}
\usepackage{mathrsfs}
\makeatletter
\numberwithin{equation}{section}
\numberwithin{figure}{section}
\theoremstyle{plain}
\newtheorem{thm}{\protect\theoremname}[section]
  \theoremstyle{plain}
  \newtheorem{lem}[thm]{\protect\lemmaname}

\usepackage{amsmath,amsfonts,amssymb,amsthm,epsfig}

\usepackage{color}

\usepackage[leqno]{amsmath}

\newcommand{\ds}{\displaystyle}

\def\R{\mathbb R}

\numberwithin{equation}{section}
\theoremstyle{definition}
\begingroup
\newtheorem{rem}[thm]{Remark}
\newtheorem{prop}[thm]{Proposition}

\endgroup

\makeatother

\usepackage{babel}
  \providecommand{\lemmaname}{Lemma}
\providecommand{\theoremname}{Theorem}

\begin{document}
\title[a Nonlinear Schr\"{o}dinger equation with
  non-symmetric electromagnetic fields]
{Infinitely many solutions for a nonlinear Schr\"{o}dinger equation with
  non-symmetric electromagnetic fields}

\author{Weiming Liu,\,\,\,\,Chunhua Wang}

\date{}
\address{[Weiming Liu] School of Mathematics and Statistics, Central China Normal University, Wuhan, 430079, P. R. China }

\email{[Weiming Liu] whu.027@163.com}

\address{[Chunhua Wang] School of Mathematics and Statistics, Central China Normal University, Wuhan, 430079, P. R. China}

\email{[Chunhua Wang] chunhuawang@mail.ccnu.edu.cn}

\begin{abstract}
In this paper, we study the nonlinear  Schr\"{o}dinger equation with non-symmetric electromagnetic fields
$$
\Big(\ds\frac{\nabla}{i}-A_{\epsilon}(x)\Big)^{2}u+V_{\epsilon}(x)u=f(u),~~~~~~u\in H^{1}(\R^{N},\mathbb{C}),
$$
where $A_{\epsilon}(x)=(A_{\epsilon,1}(x),A_{\epsilon,2}(x),\cdots,A_{\epsilon,N}(x))$ is a magnetic field satisfying that $A_{\epsilon,j}(x)(j=1,\ldots,N)$ is a real $C^{1}$  bounded function on $\mathbb{R}^{N}$ and $V_{\epsilon}(x)$
is an electric potential. Both of them  satisfy some decay conditions and $f(u)$ is a superlinear nonlinearity satisfying some non-degeneracy condition. Applying two times finite reduction methods and localized energy method, we prove that there exists some $\epsilon_{0 }> 0$ such that for $0 < \epsilon <
\epsilon_{0 }$, the above problem has infinitely many complex-valued solutions.

{\bf Keywords:} \,\,Electromagnetic fields; Finite reduction method; Localized energy method; Nonlinear Schr\"{o}dinger equation;
Non-symmetric.

{\bf Mathematics Subject Classification:}\,\,35J10, 35B99, 35J60.

\end{abstract}

\maketitle
{\small

\smallskip

\tableofcontents{}

\section{Introduction and main result}\label{s1}

In this paper, we investigate the existence of standing waves
$\psi(x,t)=e^{-\frac{iEt}{\hbar}}u(x)$, $E\in \R,u: \R^{N}\rightarrow\mathbb{C}$
to the time-dependent nonlinear Schr\"{o}dinger equation
with an external electromagnetic field
\begin{equation}\label{1.1}
i \hbar \frac{\partial \psi}{\partial t}=\Big(\ds\frac{\hbar}{i}\nabla -A(x)\Big)^{2}\psi+G(x)\psi-f(x,\psi),\,\,x\in \R^{N},
\end{equation}
which arises in various physical contexts such as nonlinear optics
 or plasma physics where one simulates the interaction effect among
 many particles by introducing a nonlinear term (see \cite{ss}). The function
 $\psi(x,t)$ takes on complex values, $\hbar$ is the Planck constant,
 $i$ is the imaginary unit. Here $A$ denotes a magnetic potential and the
  Schr\"{o}dinger operator is defined by
$$
\Big(\ds\frac{\hbar}{i}\nabla-A(x)\Big)^{2}\psi: =-\hbar^{2}\Delta \psi-\frac{2\hbar}{i}A\cdot\nabla\psi +|A|^{2}\psi-\frac{\hbar}{i}\psi div A.
$$

Actually, in general dimension, the magnetic field $B$ is a 2-form where
 $B_{k,j}=\partial_{j}A_{k}-\partial_{k}A_{j}$; in the case $N=3, B=curl A$.
 The function $G$ represents an electric potential.

  Assuming $f(x,e^{i \theta }u)=e^{i \theta}f(x,u),\theta\in \R^{1}$ and substituting
 this ansatz $\psi(x,t)=e^{-\frac{iEt}{\hbar}}u(x)$ into \eqref{1.1},
  one is led to solve the complex semilinear elliptic equation
\begin{equation}\label{1.2}
\Big(\ds\frac{\hbar}{i}\nabla-A(x)\Big)^{2}u+(G(x)-E)u=f(x,u),\,\,\,\,x\in\R^{N}.
\end{equation}

For simplicity, let $V(x)=(G(x)-E)$ and assume that $V$ is strictly positive
 on the whole space $\R^{N}.$ The transition from quantum mechanics to classical
mechanics can be formally described by letting $\hbar\rightarrow0$, and thus the existence of solutions for $\hbar$ small has physical interest.
Standing waves for $\hbar$ small are usually referred
 as semi-classical bound states (see \cite{h}).

When $A(x)\equiv 0$, problem \eqref{1.2} arises in various applications, such as chemotaxis, population genetics, chemical reactor theory, and the study of
standing waves of certain nonlinear Schr\"{o}dinger equations.
In recent years, a considerable amount of work has been devoted to study
wave solutions of \eqref{1.2} with $A(x)\equiv 0.$ Among of them, we refer to
\cite{bl1,cps,df,df1,dn,fw,l1,l2,o,r,w}. Recently, in \cite{aw}, Ao and Wei applying localized energy method obtained infinitely many positive solutions for \eqref{1.2} with non-symmetric potential.

On the contrary, there are
still relatively few papers which
deal with the case $A(x)\not\equiv 0$,
namely when a magnetic field is present. The first result on magnetic nonlinear Schr\"{o}dinger
equation is due to Esteban and Lions in \cite {el}. They obtained
the existence of standing waves to \eqref{1.2} for~$\hbar $~fixed and for special classes of
magnetic fields by solving an appropriate minimization problem for the corresponding energy functional in the cases of $N=2,3.$
In \cite{ct}, Cao and Tang constructed semiclassical multi-peak solutions for \eqref{1.2} with bounded vector potentials.
In \cite{cs1}, using a
penalization procedure, Cingolani and Secchi extended the result in \cite{cs} to the case of a vector potential $A$,
 possibly unbounded. The penalization approach was also used in \cite{bdp} by Bartsch, Dancer
and Peng to obtain multi-bump semiclassical bound for problem \eqref{1.2} with more general nonlinear term $f(x,u)$.
In \cite{k}, Kurata proved the existence of least energy solution of \eqref{1.2}
 for $\hbar>0$ under a condition relating $V(x)$ and $A(x)$. In \cite{h,h1},
Helffer studied asymptotic behavior of the eigenfunctions of the Schr\"{o}dinger
 operators with magnetic fields in the semiclassical limit. See also
 \cite{b} for generalization of the results and in \cite{hs} for potentials
 which degenerate at infinity. In \cite{lpw}, Li, Peng and Wang applied the finite reduction method to
 obtain infinitely many non-radial complex valued solutions for \eqref{1.2} with
 radial electromagnetic fields satisfying some algebraic decaying conditions.
Liu and Wang in \cite {lw} extends the result to some weaker symmetric conditions.
In \cite{pw}, Pi and Wang obtained multi-bump solutions for \eqref{1.2} with $\hbar=1,f(x,u)= |u|^{p-2}u$ and an electrical potential satisfying a
condition by applying the finite reduction method.

In this paper, inspired by \cite{aw,wy1}, our main idea is to use the Lyapunov-Schmidt reduction method. We want to point out that the only assumption we need is
the non-degeneracy of the bump. We have no requirements on the structure of the nonlinearity.

If $\hbar=1$, $A(x)=A_{0}+\epsilon \tilde{A}(x)$, $V(x)=1+\epsilon \tilde{V}(x)$ and $f(x,u)=f(u)$, then (1.2) is
reduced to the following complex problem
$$
\Big(\ds\frac{\nabla}{i}-A_{0}-\epsilon \tilde{A}(x)\Big)^{2}u+(1+\epsilon \tilde{V}(x))u=f(u),~~~~~~u\in H^{1}(\R^{N},\mathbb{C}).
$$
For simplicity of notations, in the sequel, we denote
$$
A_{\epsilon}(x)=A_{0}+\epsilon \tilde{A}(x)\,\,\,\,\,\text{and}\,\,\,\,V_{\epsilon}(x)=1+\epsilon \tilde{V}(x).
$$
Then we are concerned with the following problem
\begin{equation}\label{1.3}
\Big(\ds\frac{\nabla}{i}-A_{\epsilon}(x)\Big)^{2}u+V_{\epsilon}(x)u=f(u),~~~~~~u\in H^{1}(\R^{N},\mathbb{C}).
\end{equation}

In order to state our main result, we give the conditions imposed on $\tilde{A}(x)$, $\tilde{V}(x)$ and $f$: \\

$(A_{1})$ $\lim\limits_{|x|\rightarrow \infty} |\tilde{A}(x)|=0$; \\

$(A_{2})$ $\exists
 0<\alpha_{1}<1$,  $\lim\limits_{|x|\rightarrow \infty} |\tilde{A}(x)|^{2}e^{\alpha_{1}|x|}=+\infty$; \\

 $(A_{3})$ $\exists
 0<\alpha_{2}<1$,  $\lim\limits_{|x|\rightarrow \infty} |div \tilde{A}(x)|^{2}e^{\alpha_{2}|x|}=+\infty$; \\

 $(A_{4})$ $\lim\limits_{|x|\rightarrow \infty} |\nabla \tilde{A}(x)|=0$;\\

$(V_{1})$ $\tilde{V}(x)\in C(\R^{N},\R)$ and $\lim\limits_{|x|\rightarrow \infty} |\tilde{V}(x)|=0$; \\

$(V_{2})$ $\exists
 0<\alpha_{3}<1$,  $\lim\limits_{|x|\rightarrow \infty} |\tilde{V}(x)|e^{\alpha_{3}|x|}=+\infty$; \\

 $(f_{1})$ $f : \mathbb{C} \rightarrow \mathbb{C}$ is of class $C^{1+\delta}$ for some $0 <\delta \leq 1$,~$f'(0)=0$;\\

$(f_{2})$ $f(e^{i\theta}u)=e^{i\theta}f(u),\theta\in \R^{1};$\\

 $(f_{3})$
 The equation
\begin{eqnarray}\label{1.4}
\left\{%
\begin{array}{ll}
   \ds -\Delta w+w=f(w),& w>0 ~\text{in}~ \mathbb{R}^{N}, \vspace{0.2cm}\\
\ds \lim_{|x|\rightarrow \infty}w(x)=0,& w(0)=\ds\max_{x\in \mathbb{R}^{N}}w(x),
\end{array}
\right.
\end{eqnarray}
has a non-degenerate solution $w$, i.e.,
$$
ker(\Delta-1+f'(w))\cap L^{\infty}(\R^{N})=span\Big\{\frac{\partial w}{\partial x_{1}},\ldots,\frac{\partial w}{\partial x_{N}}\Big\}.
$$

Particularly, $f(u)=|u|^{p-1}u$ satisfies $ (f_{2}).$\\

Under the above assumptions, the spectrum of the linearized operator
$$
\Delta
\varphi -\varphi + f'(w)\varphi = \lambda\varphi, ~~\varphi \in H^{1}(\mathbb{R}^{N})
$$
admits the following decompositions
$$
\lambda_{1}>\lambda_{2}>\ldots>\lambda_{n}>\lambda_{n+1}=0>\lambda_{n+2}
$$
where each of the eigenfunction corresponding to the positive eigenvalue $\lambda_{j}$ decays exponentially.
These eigenfunctions will play an important role in our secondary Lyapunov-Schmidt
reduction(see Section 3 below).

\begin{rem}\label{rem1.4}
It is easy to find that $w$ is a solution of \eqref{1.4} if and
only if  $e^{iA_{0}\cdot x}w$ is a solution of the following problem
\begin{eqnarray}\label{ea0}
\left\{%
\begin{array}{ll}
   \ds \Big(\frac{\nabla}{i} -A_{0}\Big)^{2}u+u=f(u),& x\in\mathbb{R}^{N}, \vspace{0.2cm}\\
\ds \lim_{|x|\rightarrow \infty}|u(x)|=0,& |u(0)|=\ds\max_{x\in \mathbb{R}^{N}}|u(x)|,
\end{array}
\right.
\end{eqnarray}
from
which and $(f_{3})$ we can deduce that  \eqref{ea0} has a non-degenerate solution $e^{i \sigma+iA_{0}\cdot x}w,$
i.e.
$$
ker\Big(-\Big(\frac{\nabla}{i}-A_{0}\Big)^{2}-1+f'(w)\Big)=span\Big\{\frac{\partial (e^{i \sigma+iA_{0}\cdot x}w)}{\partial x_{1}},\ldots,\frac{\partial (e^{i \sigma+iA_{0}\cdot x}w)}{\partial x_{N}},\frac{\partial (e^{i \sigma+iA_{0}\cdot x}w)}{\partial \sigma}\Big\}.
$$

\end{rem}

In the sequel, the Sobolev space $H^{1}(\mathbb{R}^{N})$ is endowed with the standard norm
$$
\|u\|=\Bigl(\int|\nabla u|^{2}+|u|^{2}\Bigl)^{\frac{1}{2}},
$$
which is induced by the inner product
$$
\left\langle \nabla u,\nabla v\right\rangle =\int(\nabla u\nabla v+uv).
$$

Denote $\alpha=\min\{\alpha_{1},\alpha_{2},\alpha_{3}\}.$

Our main result of this paper is as follows:

\begin{thm}\label{thm1.1}
Assume that $(A_{1})$-$(A_{4})$, $(V_{1})$-$(V_{2})$ and $(f_{1})$-$(f_{3})$ hold.
Then there exists $\epsilon_{0} > 0$ such that $0 < \epsilon < \epsilon_{0}$,
problem (1.3) has infinitely many complex-valued solutions.
\end{thm}

In the following, we sketch  the main idea in the proof of Theorem 1.2.

We introduce some notations first. Let $\mu > 0$ be a real number such that $w(x) \leq ce^{-|x|}$ for $|x| > \mu$ and some constant $c$ independent of $\mu$
large. Now we define the configuration space
$$
\Omega_{1} = \mathbb{R}^{N}, \Omega_{m}:= \Bigl\{ \textbf{Q}_{m} = (Q_{1}, Q_{2},\ldots,
Q_{m})\in\mathbb{R}^{mN}:~\min_{k\neq j}|Q_{k} - Q_{j} |\geq \mu \Bigl\} , \forall m
> 1.
$$

Let $w$ be the non-degenerate solution of \eqref{1.4} and $m \geq1$ be an integer. Define the sum of $m$ spikes as
$$
w_{Q_{j}}=w(x-Q_{j}),\,\xi_{j}=e^{i\sigma+iA_{0}\cdot(x-Q_{j})},\,
z_{Q_{j}}=\xi_{j}w(x-Q_{j})\,\,\text{and}\,\,z_{ \textbf{Q}_{m}}=\sum_{j=1}^{m}z_{Q_{j}},
 $$
where $\sigma\in [0, 2\pi]$.

Let the operator be
$$
\mathcal {S}(u)=-\Bigl(\frac{\nabla}{i}-A_{\epsilon}(x)\Bigl)^{2}u-V_{\epsilon}(x)u+f(u).
$$
Fixing $(\sigma,\textbf{Q}_{m}) = (\sigma,Q_{1},\ldots, Q_{m})\in [0,2\pi]\times\Omega_{m},$ we define the following functions as the approximate kernels:
$$
D_{j,k}=\frac{\partial (e^{i\sigma+iA_{0}\cdot(x-Q_{j})}w_{Q_{j}})}{\partial x_{k}}\eta_{j}(x),~\text{for}~j=1,\ldots,m,k=1,\ldots,N$$
and
$$
D_{j,N+1}=\frac{\partial (e^{i\sigma+iA_{0}\cdot(x-Q_{j})}w_{Q_{j}})}{\partial \sigma}\eta_{j}(x),j=1,\ldots,m,
$$
where $\eta_{j}(x)=\eta(\frac{2|x-Q_{j}|}{\mu-1})$ and $\eta(t)$ is a cut off function, such that $\eta(t)=1$ for $|t|\leq1$ and $\eta(t)=0$ for $|t|\geq
\frac{\mu^{2}}{\mu^{2}-1}$. Note that the support of $D_{j,k}$ belongs to $B_{\frac{\mu^{2}}{2(\mu+1)}}(Q_{j})$.

Applying $z_{\textbf{Q}_{m}}$as the approximate solution and performing the Lyapunov-Schmidt reduction, we can show that there exists a constant $\mu_{0}$,
such that for $\mu\geq\mu_{0}$, and $\epsilon<c_{\mu}$, for some constant $c_{\mu}$ depending on $\mu$ but independent of $m$ and $\textbf{Q}_{m}$, we can find
a $\varphi_{\sigma,\textbf{Q}_{m}}$ such that
$$
\mathcal {S}(z_{\textbf{Q}_{m}}+\varphi_{\sigma,\textbf{Q}_{m}})=\sum_{j=1}^{m}\sum_{k=1}^{N+1}c_{j,k}D_{j,k},
$$
and we can show that $\varphi_{\sigma,\textbf{Q}_{m}}$ is $C^{1}$ in $(\sigma,\textbf{Q}_{m}).$ This is done in Section 2.

After that, for any $m$, we define a new function
\begin{equation}\label{m}
\mathcal {M}(\sigma,\textbf{Q}_{m})=J(z_{\textbf{Q}_{m}}+\varphi_{\sigma,\textbf{Q}_{m}}),
\end{equation}
we maximize $\mathcal
{M}(\sigma,\textbf{Q}_{m})$ over $[0,2\pi]\times \bar{\Omega}_{m}.$

At the maximum point of $\mathcal {M}(\sigma,\textbf{Q}_{m}),$ we show that $c_{j,k}=0$
for all $j,k.$ Therefore we prove that the corresponding $z_{\textbf{Q}_{m}}+\varphi_{\sigma,\textbf{Q}_{m}}$ is a solution of
\eqref{1.3}. By the arguments before, we know that there exists $\mu_{0}$ large such that $\mu\geq \mu_{0}$ and
$\epsilon\leq c_{\mu}$ and for any $m$, there exists a spike solution to \eqref{1.3} with $m$ spikes in $\Omega_{m}$.
Considering that $m$ is arbitrary, then there exists infinitely many spikes solutions for $\epsilon< c_{\mu_{0}}$ independent of $m.$

There are three main difficulties in the maximization process. Firstly, we need to show that the maximum points will not go to infinity.  Secondly, we have to
detect the difference in the energy when the spikes move to the boundary of the configuration space. In the second step, we use the induction method and
detect the difference of the m-th spikes energy and the (m+1)-th spikes energy. A crucial estimate is Lemma 3.2, where we prove that the accumulated error can be
controlled from step $m$ to step $m + 1$. To this end, we make a secondary Lyapunov-Schmidt reduction. This is done in Section 3.
Compared with \cite{aw}, since there is a magnetic filed in our problem, we have to overcome some new difficulties which involves many technical estimates.

Our paper is organized as follows. In section 2, we carry out Lyapunov-Schmidt reduction. Then we perform a second Liapunov-Schmidt reduction in section 3.
Finally, we prove our main result in section 4.

\textbf{Notations:}

1. We simply write $\int f$ to mean the Lebesgue integral of $f(x)$ in $\R^{N}.$

2. The complex conjugate of any number $z\in\mathbb{C}$ will be denoted by $\bar{z}$.

3. The real part of a number $z\in\mathbb{C}$ will be denoted by $Re z$.

4. The ordinary inner product between two vectors $a,b\in \R^{N}$ will be denoted by $a\cdot b$.

{\bf Acknowledgements:}
  This paper was partially supported by NSFC (No.11301204; No.11371159), self-determined research funds of CCNU from
  the colleges' basic research and operation of MOE (CCNU14A05036).

\section{Finite-dimensional reduction}\label{s2}

In this section, we perform a finite-dimensional reduction.

Let $\gamma\in(0, 1)$ and we define
\begin{equation}\label{2.1}
E(\cdot):= \sum_{j=1}^{m} e^{-\gamma|\cdot-Q_{j}|},\,\,\,\,\text{where}\,\,\,\, \textbf{Q}_{m}\in\Omega_{m}.
\end{equation}

Consider the norm
\begin{equation}\label{2.2}
 \|f\|_{*}= \sup_{x\in\mathbb{R}^{N}}| E(x)^{-1}f(x)|,
\end{equation}
which was first introduced in \cite{mpw} and also used in \cite{aw,wy1}.
Now we investigate
\begin{eqnarray}\label{2.3}
\left\{%
\begin{array}{ll}
   L(\varphi_{\sigma,\textbf{Q}_{m}}):=-\Bigl(\ds \frac{\nabla}{i}-A_{\epsilon}(x)\Bigl)^{2}\varphi_{\sigma,\textbf{Q}_{m}}-V_{\epsilon}(x)\varphi_{\sigma,\textbf{Q}_{m}}
   +f'(z_{\textbf{Q}_{m}})\varphi_{\sigma,\textbf{Q}_{m}} \vspace{0.2cm}\\
=h+\ds\sum_{j=1}^{m}\sum_{k=1}^{N+1}c_{j,k}D_{j,k},~in~\mathbb{R}^{N},\vspace{0.2cm}\\
\ds Re\int\varphi_{\sigma,\textbf{Q}_{m}}\bar{D}_{j,k}=0~for~j=1,\ldots,m,k=1,\ldots,N+1.
\end{array}
\right.
\end{eqnarray}

Firstly, we give a result which will be used later.
\begin{lem}\label{number}(\cite{dwy}, Lemma 3.4)
There exists a constant $C_{N}=6^{N}$ such that for any $m\in \mathbb{N}^{+}$ and any
$\textbf{Q}_{m}=(Q_{1},Q_{2},...,Q_{m})\in \R^{mN},$
\begin{equation}\label{s1.1}
\sharp\Big\{Q_{j}\Big| \frac{l}{2}\mu\leq|x-Q_{j}|<\frac{(l+1)}{2}\mu\Big\}\leq C_{N}(l+1)^{N-1}
\end{equation}
for all $x\in \R^{N}$ and all $l\in \mathbb{N}.$ Particularly, we have
\begin{equation}\label{s1.2}
\sharp\Big\{Q_{j}\Big|0\leq|x-Q_{j}|<\frac{\mu}{2}\Big\}\leq C_{N}.
\end{equation}
\end{lem}

\begin{lem}\label{lem2.1}
Let $h$ with $\|h\|_{*}$ bounded and assume that $(\varphi_{\sigma,\textbf{Q}_{m}}, {c_{j,k}})$ is a solution to \eqref{2.3}. Then there exist positive numbers $\mu_{0}$
and $C$, such that for all $0 < \epsilon < e^{-2\mu},$ $\mu>\mu_{0}$ and $(\sigma,\textbf{Q}_{m})\in [0, 2\pi]\times \Omega_{m},$  one has
\begin{equation}\label{2.4}
 \|\varphi_{\sigma,\textbf{Q}_{m}}\|_{*}\leq
C\|h\|_{*},
\end{equation}
where $C$ is a positive constant independent of $\mu,m$ and $\textbf{Q}_{m}\in\Omega_{m}$.
\end{lem}

\begin{proof}
We prove it by contradiction. Assume that there exists a solution $\varphi_{\sigma,\textbf{Q}_{m}}$ to \eqref{2.3}  and $\|h\|_{*}\rightarrow0$,
$\|\varphi_{\sigma,\textbf{Q}_{m}}\|_{*}=1$.

Multiplying the equation in \eqref{2.3}  by $\bar{D}_{j,k}$ and integrating in $\mathbb{R}^{N}$, we get
\begin{equation}\label{2.5}
Re\int L(\varphi_{\sigma,\textbf{Q}_{m}})\bar{D}_{j,k}=Re\int h\bar{D}_{j,k}+c_{j,k}\int|D_{j,k}|^{2}.
\end{equation}
Considering the exponential decay at infinity of $\frac{\partial w(x)}{\partial x_{k}}$ and the definition of $D_{j,k}(k=1,\ldots,N+1)$, we have
\begin{equation}\label{2.6}
\begin{array}{ll}
&\ds\int|D_{j,k}|^{2}=\ds\int\Bigl|\bigl(iA_{0,k}z_{Q_{j}}+\frac{\partial w_{Q_{j}}}{\partial x_{k}}\xi_{j}\bigl)\eta_{j}\Bigl|^{2}\vspace{0.2cm}\\
%&=\ds\int|A_{0}|^{2}\xi_{j}^{2}w_{Q_{j}}^{2}\eta_{j}^{2}+\int \xi_{j}^{2}\Bigl(\frac{\partial w_{Q_{j}}}{\partial x_{k}}\Bigl)^{2}\eta_{j}^{2}+2Re\int iA_{0}z_{Q_{j}}\eta_{j}^{2}\frac{\partial w_{Q_{j}}}{\partial x_{k}}\bar{\xi}_{j}      \vspace{0.2cm}\\
&\ds
=\int A_{0,k}^{2}w_{Q_{j}}^{2}\eta^{2}\Bigl(\frac{2|x-Q_{j}|}{\mu-1}\Bigl)+\int \Bigl(\frac{\partial w_{Q_{j}}}{\partial x_{k}}\Bigl)^{2}\eta^{2}\Bigl(\frac{2|x-Q_{j}|}{\mu-1}\Bigl)\vspace{0.2cm}\\
&=
\ds A_{0,k}^{2}\int w^{2}+A_{0,k}^{2}\int_{B^{C}_{\frac{\mu-1}{2}}(0)}\Bigl[\eta^{2}\Bigl(\frac{2|x|}{\mu-1}\Bigl)-1\Bigl]w^{2}
\vspace{0.2cm}\\
&\quad+\ds\int
\Bigl(\frac{\partial w}{\partial x_{k}}\Bigl)^{2}+\ds\int_{B^{C}_{\frac{\mu-1}{2}}(0)}\Bigl[\eta^{2}\Bigl(\frac{2|x|}{\mu-1}\Bigl)-1\Bigl)\Bigl]\Bigl(\frac{\partial w}{\partial x_{k}}\Bigl)^{2}\vspace{0.2cm}\\
%&\ds\leq|A_{0}|^{2}\int w^{2}+|A_{0}|^{2}\int\Bigl(\eta^{2}\Bigl(\frac{2|x|}{\mu-1}\Bigl)-1\Bigl)w^{2}+\int
%\Bigl(\frac{\partial w}{\partial x_{k}}\Bigl)^{2}+C\int^{+\infty}_{\frac{\mu-1}{2}}\int_{\partial B_{r}(0)}r^{-(N-1)}e^{-2r}dSdr\vspace{0.2cm}\\
%&\ds\leq|A_{0}|^{2}\int w^{2}+O(e^{-\mu})+\int \Bigl(\frac{\partial w}{\partial x_{k}}\Bigl)^{2}
%+C\int^{+\infty}_{\frac{\mu-1}{2}}e^{-2r}dSdr\vspace{0.2cm}\\
&=
 A_{0,k}^{2}\ds\int w^{2}+\int \Bigl(\frac{\partial w}{\partial x_{k}}\Bigl)^{2}+O(e^{-\mu}),~as~\mu\rightarrow+\infty,~k=1,2,\ldots,N
\end{array}
\end{equation}
and
\begin{equation}\label{2.7}
\begin{array}{ll}
\ds\int|D_{j,N+1}|^{2}&=\ds\int\Bigl|iz_{Q_{j}}\eta\Bigl(\frac{2|x-Q_{j}|}{\mu-1}\Bigl)\Bigl|^{2}=\int\Bigl|w_{Q_{j}}\eta\Bigl(\frac{2|x-Q_{j}|}{\mu-1}\Bigl)\Bigl|^{2}\vspace{0.2cm}\\
&
=\ds\int
w^{2}+\int_{B^{C}_{\frac{\mu-1}{2}}(0)}\Bigl[\eta^{2}\Bigl(\frac{2|x|}{\mu-1}\Bigl)-1\Bigl]w^{2}
=\ds\int w^{2}+O(e^{-\mu}),~as~\mu\rightarrow+\infty.
\end{array}
\end{equation}

On the other hand, by Lemma \ref{lemw} we have
\begin{equation}\label{2.8}
\begin{array}{ll}
\ds\Bigl|Re\int h\bar{D}_{j,k}\Bigl|
&=\ds\Bigl|Re\int h\bigl(-iA_{0,k}\bar{z}_{Q_{j}}+\frac{\partial w_{Q_{j}}}{\partial x_{k}}\bar{\xi}_{j}\bigl)\eta_{j}\Bigl|        \vspace{0.2cm}\\
&\leq
\ds\int|h||A_{0,k}|w_{Q_{j}}|\eta_{j}|+\int|h|\Big|\frac{\partial w_{Q_{j}}}{\partial x_{k}}\Big||\eta_{j}|\vspace{0.2cm}\\
&\leq
C\|h\|_{*}\ds\int_{B_{\frac{\mu^{2}}{2(\mu+1)}}(Q_{j})}|A_{0,k}|\sum_{j=1}^{m}e^{-\gamma|x-Q_{j}|}w(x-Q_{j})
\Bigl|\eta\Bigl(\frac{2|x-Q_{j}|}{\mu-1}\Bigl)\Bigl|\vspace{0.2cm}\\
&\ds\,\,\,\,\,\,+C\|h\|_{*}\int_{B_{\frac{\mu^{2}}{2(\mu+1)}}(Q_{j})}\sum_{j=1}^{m} e^{-\gamma|x-Q_{j}|}\Bigl|\frac{\partial w(x-Q_{j})}{\partial x_{k}}\Bigl|\Bigl|\eta\Bigl(\frac{2|x-Q_{j}|}{\mu-1}\Bigl)\Bigl|\vspace{0.2cm}\\
\end{array}
\end{equation}
\begin{equation*}
\begin{array}{ll}
&\leq
C\|h\|_{*}\ds\int_{B_{\frac{\mu}{2}}(Q_{j})}e^{-\gamma|x-Q_{j}|}w(x-Q_{j})+C\|h\|_{*}\int_{B_{\frac{\mu}{2}}(Q_{j})} e^{-\gamma|x-Q_{j}|}\Bigl|\frac{\partial w(x-Q_{j})}{\partial x_{k}}\Bigl|\vspace{0.2cm}\\
&\leq
C\|h\|_{*}\ds\int_{0}^{\frac{\mu}{2}}e^{-(1+\gamma) t}t^{N-1}dt
\ds\leq C\|h\|_{*},\quad\quad\quad\quad k=1,2,\ldots,N
\end{array}
\end{equation*}
and
\begin{equation}\label{2.9}
\begin{array}{ll}
\ds\Bigl|Re\int h\bar{D}_{j,N+1}\Bigl|&\ds\leq\int|h|||\bar{D}_{j,N+1}|\leq\int|h|||iz_{Q_{j}}\eta_{j}|\vspace{0.2cm}\\
&\leq C\|h\|_{*}\ds\int\sum_{j=1}^{m} e^{-\gamma|x-Q_{j}|}w(x-Q_{j})\eta\Bigl(\frac{2|x-Q_{j}|}{\mu-1}\Bigl)
\vspace{0.2cm}\\
&\leq C\|h\|_{*}\ds\int_{B_{\frac{\mu}{2}}(Q_{j})} e^{-\gamma|x-Q_{j}|}w(x-Q_{j})
\vspace{0.2cm}\\
&\leq C\|h\|_{*}\ds\int_{0}^{\frac{\mu}{2}}e^{-(1+\gamma) t}t^{N-1}dt
\leq C\|h\|_{*}.
\end{array}
\end{equation}

Here and in what follows, $C$ stands for a positive constant independent of $\epsilon$ and $\mu$, as $\epsilon\rightarrow0$. Now if we write
$\widetilde{D}_{j,k} = \frac{\partial (e^{i\sigma+iA_{0}\cdot(x-Q_{j})}w_{Q_{j}})}{ \partial x_{k}}$ , then we have
\begin{equation}\label{2.10}
\begin{array}{ll}
&\ds Re\int L(\varphi_{\sigma,\textbf{Q}_{m}})\bar{D}_{j,k}=Re\int L(D_{j,k})\bar{\varphi}_{\sigma,\textbf{Q}_{m}}\vspace{0.2cm}\\
&=
Re\ds\int\Bigl[-\Bigl(\frac{\nabla}{i}-A_{\epsilon}(x)\Bigl)^{2}D_{j,k}\bar{\varphi}_{\sigma,\textbf{Q}_{m}}-V_{\epsilon}(x) D_{j,k}\bar{\varphi}_{\sigma,\textbf{Q}_{m}}+f'(z_{\textbf{Q}_{m}})D_{j,k}\bar{\varphi}_{\sigma,\textbf{Q}_{m}}\Bigl]\vspace{0.2cm}\\
&=
Re\ds\int\Bigl[-\Bigl(\frac{\nabla}{i}-A_{0}\Bigl)^{2}
D_{j,k}\bar{\varphi}_{\sigma,\textbf{Q}_{m}}-D_{j,k}\bar{\varphi}_{\sigma,\textbf{Q}_{m}}+f'(z_{Q_{j}})D_{j,k}\bar{\varphi}_{\sigma,\textbf{Q}_{m}}\Bigl]\vspace{0.2cm}\\
&\,\,\,\,\,\,
-Re\ds\int\epsilon \tilde{V} D_{j,k}\bar{\varphi}_{\sigma,\textbf{Q}_{m}}+Re\int[f'(z_{\textbf{Q}_{m}})
-f'(z_{Q_{j}})]D_{j,k}\bar{\varphi}_{\sigma,\textbf{Q}_{m}}\vspace{0.2cm}\\
&\,\,\,\,\,\,
+Re\ds\int\Bigl(\frac{\epsilon}{i}div\tilde{A}-2\epsilon A_{0}\cdot \tilde{A}-\epsilon^{2}|\tilde{A}|^{2}\Bigl)D_{j,k}\bar{\varphi}_{\sigma,\textbf{Q}_{m}}+Re\int\frac{2\epsilon}{i}\tilde{A}(x)\cdot\nabla D_{j,k}\bar{\varphi}_{\sigma,\textbf{Q}_{m}}\vspace{0.2cm}\\
&\leq
Re\ds\int_{B_{\frac{\mu^{2}}{2(\mu+1)}}(Q_{j})}\Bigl[-\Bigl(\frac{\nabla}{i}-A_{0}\Bigl)^{2}
\widetilde{D}_{j,k}-\widetilde{D}_{j,k}+f'(z_{Q_{j}})\widetilde{D}_{j,k}\Bigl]\eta_{j}\bar{\varphi}_{\sigma,\textbf{Q}_{m}}\vspace{0.2cm}\\
&\,\,\,\,\,\,
+Re\ds\int_{B_{\frac{\mu^{2}}{2(\mu+1)}}(Q_{j})\backslash B_{\frac{\mu-1}{2}}(Q_{j})}
\Bigl[\widetilde{D}_{j,k}\Delta\eta_{j}+2\nabla\eta_{j}\cdot\nabla\widetilde{D}_{j,k}
+\frac{2}{i}A_{0}\cdot\nabla\eta_{j}\widetilde{D}_{j,k}\Bigl]\bar{\varphi}_{\sigma,\textbf{Q}_{m}}\vspace{0.2cm}\\
&\,\,\,\,\,\,
-Re\ds\int_{B_{\frac{\mu^{2}}{2(\mu+1)}}(Q_{j})}\epsilon \tilde{V}
\widetilde{D}_{j,k}\eta_{j}\bar{\varphi}_{\sigma,\textbf{Q}_{m}}+Re\int_{B_{\frac{\mu^{2}}{2(\mu+1)}}(Q_{j})}[f'(z_{\textbf{Q}_{m}})
-f'(z_{Q_{j}})]\widetilde{D}_{j,k}\eta_{j}\bar{\varphi}_{\sigma,\textbf{Q}_{m}}\vspace{0.2cm}\\
&\,\,\,\,\,\,
+Re\ds\int_{B_{\frac{\mu^{2}}{2(\mu+1)}}(Q_{j})}\Bigl(\frac{\epsilon}{i}div\tilde{A}-2\epsilon A_{0}\cdot
\tilde{A}-\epsilon^{2}|\tilde{A}|^{2}\Bigl)\widetilde{D}_{j,k}\eta_{j}\bar{\varphi}_{\sigma,\textbf{Q}_{m}}\\
&\,\,\,\,\,\,
+Re\ds\int_{B_{\frac{\mu^{2}}{2(\mu+1)}}(Q_{j})}\frac{2\epsilon}{i}\tilde{A}(x)\cdot\nabla
\widetilde{D}_{j,k}\eta_{j}\bar{\varphi}_{\sigma,\textbf{Q}_{m}}
+Re\ds\int_{B_{\frac{\mu^{2}}{2(\mu+1)}}(Q_{j})}\frac{2\epsilon}{i}\tilde{A}(x)\cdot\nabla \eta_{j}\widetilde{D}_{j,k}\bar{\varphi}_{\sigma,\textbf{Q}_{m}}.
\end{array}
\end{equation}

Since $$-\Bigl(\frac{\nabla}{i}-A_{0}\Bigl)^{2} \widetilde{D}_{j,k}-\widetilde{D}_{j,k}+f'(z_{Q_{j}})\widetilde{D}_{j,k}=0,$$ we have
\begin{equation}\label{2.11}
Re\int_{B_{\frac{\mu^{2}}{2(\mu+1)}}(Q_{j})}\Bigl[-\Bigl(\frac{\nabla}{i}-A_{0}\Bigl)^{2}
\widetilde{D}_{j,k}-\widetilde{D}_{j,k}+f'(z_{Q_{j}})\widetilde{D}_{j,k}\Bigl]\eta_{j}\bar{\varphi}_{\sigma,\textbf{Q}_{m}}=0.
\end{equation}
Moreover, by Lemma \ref{lemw} we have
\begin{equation}\label{2.12}
\begin{array}{ll}
&\ds\Bigl|Re\int_{B_{\frac{\mu^{2}}{2(\mu+1)}}(Q_{j})\backslash B_{\frac{\mu-1}{2}}(Q_{j})}\Bigl(\widetilde{D}_{j,k}\Delta\eta_{j}+2\nabla\eta_{j}\cdot\nabla\widetilde{D}_{j,k}
+\frac{2}{i}A_{0}\cdot\nabla\eta_{j}\widetilde{D}_{j,k}\Bigl)\bar{\varphi}_{\sigma,\textbf{Q}_{m}}\Bigl|
\vspace{0.2cm}\\
&\leq
C\|\varphi_{\sigma,\textbf{Q}_{m}}\|_{*}\ds\int_{B_{\frac{\mu^{2}}{2(\mu+1)}}(Q_{j})\backslash B_{\frac{\mu-1}{2}}(Q_{j})}
\sum_{j=1}^{m}e^{-\gamma|x-Q_{j}|}\Big(\Big|\frac{\partial w_{Q_{j}}}{\partial x_{k}}\Big|+w_{Q_{j}}+\Big|\nabla w_{Q_{j}}\Big|+\Big|\nabla \frac{\partial w_{Q_{j}}}{\partial x_{k}}\Big|\Big)
\vspace{0.2cm}\\
&\leq
C\|\varphi_{\sigma,\textbf{Q}_{m}}\|_{*}\ds\int^{\frac{\mu^{2}}{2(\mu+1)}}_{\frac{\mu-1}{2}}e^{-(1+\gamma)s}s^{N-1}ds\leq
Ce^{-(1+\beta)\frac{\mu}{2}}\|\varphi_{\sigma,\textbf{Q}_{m}}\|_{*}
\end{array}
\end{equation}
for some $\beta>0$.

Observing that
$$
\Bigl|f'(z_{\textbf{Q}_{m}}) -f'(z_{Q_{j}})\Bigl|\leq C\Bigl|\sum_{k\neq j}z_{Q_{k}}\Bigl|^{\delta},
$$
by $(f_{1})$ we have
\begin{equation}\label{2.13}
\begin{array}{ll}
&\ds \Bigl|Re\int_{B_{\frac{\mu^{2}}{2(\mu+1)}}(Q_{j})}(f'(z_{\textbf{Q}_{m}}) -f'(z_{Q_{j}}))\widetilde{D}_{j,k}\eta_{j}\bar{\varphi}_{\sigma,\textbf{Q}_{m}}\Bigl|\vspace{0.2cm}\\
&\ds\leq C\|\varphi_{\sigma,\textbf{Q}_{m}}\|_{*}\int_{B_{\frac{\mu^{2}}{2(\mu+1)}}(Q_{j})}\Bigl|\sum_{k\neq
j}z_{Q_{k}}\Bigl|^{\delta}\sum_{j=1}^{m} e^{-\gamma|x-Q_{j}|}\Big(\Big|\frac{\partial w_{Q_{j}}}{\partial x_{k}}\Big|+w_{Q_{j}}\Big)\vspace{0.2cm}\\
&\ds\leq C\|\varphi_{\sigma,\textbf{Q}_{m}}\|_{*}\int_{B_{\frac{\mu^{2}}{2(\mu+1)}}(Q_{j})}\sum_{k\neq
j}|w_{Q_{k}}|^{\delta}\sum_{j=1}^{m} e^{-\gamma|x-Q_{j}|}\Big(\Big|\frac{\partial w_{Q_{j}}}{\partial x_{k}}\Big|+w_{Q_{j}}\Big)\vspace{0.2cm}\\
&\ds\leq C\|\varphi_{\sigma,\textbf{Q}_{m}}\|_{*}\int_{B_{\frac{\mu^{2}}{2(\mu+1)}}(Q_{j})}e^{-\frac{\delta}{2}\mu}\sum_{j=1}^{m} e^{-\gamma|x-Q_{j}|}\Big(\Big|\frac{\partial w_{Q_{j}}}{\partial x_{k}}\Big|+w_{Q_{j}}\Big)\vspace{0.2cm}\\
&\ds\leq C\|\varphi_{\sigma,\textbf{Q}_{m}}\|_{*}e^{-\frac{\delta}{2}\mu}\int^{\frac{\mu^{2}}{2(\mu+1)}}_{0}e^{-(1+\gamma)s}s^{N-1}ds\leq
Ce^{-\beta\frac{\mu}{2}}\|\varphi_{\sigma,\textbf{Q}_{m}}\|_{*}
\end{array}
\end{equation}
and
\begin{equation}\label{2.14}
\begin{array}{ll}
&\ds \Bigl|Re\int_{B_{\frac{\mu^{2}}{2(\mu+1)}}(Q_{j})}\epsilon \tilde{V} \widetilde{D}_{j,k}\eta_{j}\bar{\varphi}_{\sigma,\textbf{Q}_{m}}\Bigl|\leq\epsilon
\int_{B_{\frac{\mu^{2}}{2(\mu+1)}}(Q_{j})}
|\tilde{V}| |\widetilde{D}_{j,k}||\bar{\varphi}_{\sigma,\textbf{Q}_{m}}|\vspace{0.2cm}\\
&\ds\leq Ce^{-2\mu}\|\varphi_{\sigma,\textbf{Q}_{m}}\|_{*}\int_{B_{\frac{\mu^{2}}{2(\mu+1)}}(Q_{j})}|\tilde{V}| |\widetilde{D}_{j,k}|\sum_{j=1}^{m} e^{-\gamma|x-Q_{j}|}\vspace{0.2cm}\\
&\ds\leq Ce^{-2\mu}\|\varphi_{\sigma,\textbf{Q}_{m}}\|_{*}\int_{B_{\frac{\mu^{2}}{2(\mu+1)}}(Q_{j})}\Big(\Big|\frac{\partial w_{Q_{j}}}{\partial x_{k}}\Big|+w_{Q_{j}}\Big)\sum_{j=1}^{m} e^{-\gamma|x-Q_{j}|}\vspace{0.2cm}\\
&\ds\leq C\|\varphi_{\sigma,\textbf{Q}_{m}}\|_{*}e^{-\frac{\delta}{2}\mu}\int^{\frac{\mu^{2}}{2(\mu+1)}}_{0}e^{-(1+\gamma)s}s^{N-1}ds\leq
Ce^{-\beta\frac{\mu}{2}}\|\varphi_{\sigma,\textbf{Q}_{m}}\|_{*}.
\end{array}
\end{equation}

Similarly, we can get
\begin{equation}\label{2.15}
\Bigl|Re\int_{B_{\frac{\mu^{2}}{2(\mu+1)}}(Q_{j})}\frac{\epsilon}{i}div\tilde{A} \widetilde{D}_{j,k}\eta_{j}\bar{\varphi}_{\sigma,\textbf{Q}_{m}}\Bigl|\leq
C\epsilon\int_{B_{\frac{\mu^{2}}{2(\mu+1)}}(Q_{j})}|\widetilde{D}_{j,k}||\bar{\varphi}_{\sigma,\textbf{Q}_{m}}|\leq
Ce^{-\beta\frac{\mu}{2}}\|\varphi_{\sigma,\textbf{Q}_{m}}\|_{*},
\end{equation}

\begin{equation}\label{2.16}
\Bigl|Re\int_{B_{\frac{\mu^{2}}{2(\mu+1)}}(Q_{j})}2\epsilon A_{0}\cdot \tilde{A }\widetilde{D}_{j,k}\eta_{j}\bar{\varphi}_{\sigma,\textbf{Q}_{m}}\Bigl|\leq
C\epsilon\int_{B_{\frac{\mu^{2}}{2(\mu+1)}}(Q_{j})}|\widetilde{D}_{j,k}||\bar{\varphi}_{\sigma,\textbf{Q}_{m}}|\leq
Ce^{-\beta\frac{\mu}{2}}\|\varphi_{\sigma,\textbf{Q}_{m}}\|_{*},
\end{equation}

\begin{equation}\label{2.17}
\Bigl|Re\int_{B_{\frac{\mu^{2}}{2(\mu+1)}}(Q_{j})}\epsilon^{2}|\tilde{A}|^{2}\widetilde{D}_{j,k}\eta_{j}\bar{\varphi}_{\sigma,\textbf{Q}_{m}}\Bigl|\leq
C\epsilon\int_{B_{\frac{\mu^{2}}{2(\mu+1)}}(Q_{j})}|\widetilde{D}_{j,k}||\bar{\varphi}_{\sigma,\textbf{Q}_{m}}|\leq
Ce^{-\beta\frac{\mu}{2}}\|\varphi_{\sigma,\textbf{Q}_{m}}\|_{*},
\end{equation}

\begin{equation}\label{2.18}
\Bigl|Re\int_{B_{\frac{\mu^{2}}{2(\mu+1)}}(Q_{j})}\frac{2\epsilon}{i}\tilde{A}(x)\cdot\nabla\eta_{j}\widetilde{D}_{j,k}\bar{\varphi}_{\sigma,\textbf{Q}_{m}}\Bigl|\leq
C\epsilon\int_{B_{\frac{\mu^{2}}{2(\mu+1)}}(Q_{j})}|\widetilde{D}_{j,k}||\bar{\varphi}_{\sigma,\textbf{Q}_{m}}|\leq
Ce^{-\beta\frac{\mu}{2}}\|\varphi_{\sigma,\textbf{Q}_{m}}\|_{*}
\end{equation}
and
\begin{equation}\label{2.19}
\Bigl|Re\int_{B_{\frac{\mu^{2}}{2(\mu+1)}}(Q_{j})}\frac{2\epsilon}{i}\tilde{A}(x)\cdot\nabla\widetilde{D}_{j,k}\eta_{j}\bar{\varphi}_{\sigma,\textbf{Q}_{m}}\Bigl|\leq
C\epsilon\int_{B_{\frac{\mu^{2}}{2(\mu+1)}}(Q_{j})}|\nabla\widetilde{D}_{j,k}||\bar{\varphi}_{\sigma,\textbf{Q}_{m}}|\leq
Ce^{-\beta\frac{\mu}{2}}\|\varphi_{\sigma,\textbf{Q}_{m}}\|_{*},
\end{equation}
for some $\beta>0$.

It follows from \eqref{2.5} to \eqref{2.19} that
\begin{equation}\label{2.20}
|c_{j,k}| \leq C(e^{-\beta\frac{\mu}{2}}\|\varphi_{\sigma,\textbf{Q}_{m}}\|_{*}+\|h\|_{*}).
\end{equation}

 Let now $\theta\in(0, 1)$. It is easy to check that the function $E(x)$ in \eqref{2.1} satisfies
\begin{equation}\label{l}
|-L(E(x))| \geq\frac{1}{2}(1-\theta^{2})E(x), ~in ~\mathbb{R}^{N}\backslash\bigcup_{j=1}^{m}B_{\bar{\mu}}(Q_{j})
 \end{equation}
 provided $\bar{\mu}$ is large enough and $\bar{\mu}\leq \frac{\mu}{2}.$
Indeed, by Lemma \ref{number}  we have
\begin{eqnarray*}
w_{\textbf{Q}_{m}}&\leq&\sum_{|x-Q_{j}|<\frac{1}{2}\mu}w(x-Q_{j})
+\sum_{l=1}^{\infty}\sum_{\frac{l}{2}\mu\leq|x-Q_{j}|<\frac{l+1}{2}\mu}w(x-Q_{j})\\
&\leq&Cw(\bar{\mu})+C\sum_{l=1}^{\infty}l^{N-1}e^{-\frac{l}{2}\mu}
\leq Cw(\bar{\mu}).
\end{eqnarray*}
Then
\begin{equation}\label{ll1}
|f'(z_{\textbf{Q}_{m}})|\leq C(w_{\textbf{Q}_{m}})^{\delta}
\leq Cw^{\delta}(\bar{\mu})\leq \frac{1-\theta^{2}}{4},\,\,\,\text{in}\,\,\,\R^{N}\backslash \cup_{j=1}^{m}B_{\bar{\mu}}(Q_{j}).
\end{equation}
 From \eqref{ll1} and direct computation, we have
\begin{eqnarray*}
|-L(E(x))|&=&\Big|\Big(\frac{\nabla}{i}-A_{\epsilon}(x)\Big)^{2}E(x)+V_{\epsilon}E(x)-
f'(z_{\textbf{Q}_{m}})E(x)\Big|
\\
&=&\sum_{j=1}^{m} \Big|-\gamma^{2}+1+|A_{0}|^{2}+\gamma\frac{(N-1)-2iA_{\epsilon}\cdot (1,1,...,1)}{|x-Q_{j}|}
\\
&&\quad\quad
+\epsilon (\tilde{V}(x)+2A_{0}\cdot \tilde{A}+idiv\tilde{A}+\epsilon|\tilde{A}|^{2})-f'(z_{\textbf{Q}_{m}})\Big|e^{-\gamma|x-Q_{j}|}\\
&\geq & \frac{1}{2}(1-\theta^{2})E(x),\,\,\,\,\text{in}\,\,\,\R^{N}\backslash \cup_{j=1}^{m}B_{\bar{\mu}}(Q_{j}),
\end{eqnarray*}
which yields that \eqref{l} is true.

 Hence the function $E(x)$ can be
used as a barrier to prove the pointwise estimate
\begin{equation}\label{2.21}
|\varphi_{\sigma,\textbf{Q}_{m}}(x)| \leq C\bigl(\|L\varphi_{\sigma,\textbf{Q}_{m}}\|_{*}+\sup_{j}\|\varphi_{\textbf{Q}_{m}(x)}\|_{L^{\infty}(\partial
B_{\bar{\mu}}(Q_{j}))}\bigl)E(x),
\end{equation}
for all $x\in\mathbb{R}^{N}\backslash\bigcup_{j=1}^{m}B_{\bar{\mu}}(Q_{j})$.

Now we prove it by contradiction. We assume that there exist a sequence of $\epsilon$ tending to $0$, $\mu$ tending to $\infty$ and a sequence of solutions
of \eqref{2.3} for which the inequality is not true. The problem being linear, we can reduce to the case where we have a sequence $\epsilon^{(n)}$ tending to $0$,
$\mu^{(n)}$ tending to $\infty$ and sequences $h^{(n)}, \varphi^{(n)}, {c^{(n)}_{j,k}}$ such that
$$
\|h^{(n)}\|\rightarrow0\,\,\,\text{and} \,\,\,
\|\varphi_{\sigma,\textbf{Q}_{m}}^{(n)}\|_{*} = 1.
$$

By \eqref{2.20},we have $$\Bigl\|\sum_{jk}c^{(n)}_{jk}D_{jk}\Bigl\|_{*}\rightarrow0.$$

Then \eqref{2.21} implies that there exists $Q^{(n)}_{j}\in\Omega_{m}$ such that
\begin{equation}\label{2.22}
\|\varphi_{\textbf{Q}_{m}^{(n)}}\|_{L^{\infty}(B_{\frac{\mu}{2}}(Q_{j}^{(n)}))}\geq C
\end{equation}
for some fixed constant $C > 0$.

Applying elliptic estimates together with Ascoli-Arzela's theorem, we can find a sequence
$Q^{(n)}_{j}$ and we can extract, from the sequence $\varphi^{(n)}(\cdot-Q^{(n)}_{j})$ a subsequence which will converge to $\varphi^{\infty}$ a solution of
$$\Bigl[-\bigl(\frac{\nabla}{i}-A_{0}\bigl)^{2}-1+f'(e^{i\sigma+iA_{0}\cdot x}w)\Bigl]\varphi^{\infty} = 0,~ x\in~ \mathbb{R}^{N},$$
which is bounded by a constant times $e^{-\gamma|x|}$, with $\gamma> 0$. Moreover, recall that $\varphi^{(n)}_{\textbf{Q}_{m}}$ satisfies the orthogonality
conditions in \eqref{2.3}. Therefore, the limit function $\varphi^{\infty}$ also satisfies $$Re\int\varphi^{\infty}\overline{\frac{\partial z}{\partial x_{j}}}=0,~j=1,\ldots,N,~\text{~and~}Re\int\varphi^{\infty}\overline{\frac{\partial z}{\partial \sigma}}=0,$$
where $z=e^{i\sigma+iA_{0}\cdot x}w(x)$.

Then we have
that $\varphi^{\infty}\equiv0$ which contradicts to \eqref{2.22}.
\end{proof}

From Lemma \ref{lem2.1}, we can obtain the following result
\begin{prop}\label{prop2.2}
Then there exist positive numbers $\gamma\in(0, 1)$, $\mu_{0}> 0$ and $C > 0$, such that for all $0 < \epsilon< e^{-2\mu}, \mu > \mu_{0}$ and for any given
$h$ with $\|h\|_{*}$ norm bounded, there is a unique solution $(\varphi_{\sigma,\textbf{Q}_{m}}, {c_{j,k}})$ to problem \eqref{2.3}. Moreover,
\begin{equation}\label{2.23}
\|\varphi_{\sigma,\textbf{Q}_{m}}\|_{*}\leq C\|h\|_{*}.
\end{equation}
\end{prop}

\begin{proof}
Here we consider the space $$\mathcal {H}= \Bigl\{ u \in H^{1}(\mathbb{R}^{N}) : Re\int u\bar{D}_{j,k} = 0, (\textbf{Q}_{m},\sigma)\in\Omega_{m}\times[0,2\pi]\Bigl\} .$$

Problem \eqref{2.3} can be rewritten as
\begin{equation}\label{2.24}
\varphi_{\sigma,\textbf{Q}_{m}} + \mathcal {K}(\varphi_{\sigma,\textbf{Q}_{m}}) = \bar{h}, ~in~ \mathcal {H},
\end{equation}
where $\bar{h}$ is defined by duality and $\mathcal {K}: H \rightarrow H$ is a linear compact operator. By Fredholm's alternative, we know that the equation
\eqref{2.24} has a unique solution for $\bar{h}=0$ which in turn follows from Lemma \ref{lem2.1}. The estimate \eqref{2.23} follows from directly from \eqref{2.4} in Lemma \ref{lem2.1}. The
proof is complete.
\end{proof}

In the sequel, if $\varphi_{\sigma,\textbf{Q}_{m}}$ is the unique solution given by Proposition \ref{prop2.2}, we denote
\begin{equation}\label{2.25}
\varphi_{\sigma,\textbf{Q}_{m}} = \mathcal {A}(h).
\end{equation}
By \eqref{2.23}, we have
\begin{equation}\label{2.26}
\|\mathcal {A}(h)\|_{*}\leq C\|h\|_{*}.
\end{equation}

Now, we consider
\begin{eqnarray}\label{2.27}
\left\{%
\begin{array}{ll}
   \ds -\Bigl(\frac{\nabla}{i}-A_{\epsilon}(x)\Bigl)^{2}(z_{\textbf{Q}_{m}}+\varphi_{\sigma,\textbf{Q}_{m}})-V_{\epsilon}(x)(z_{\textbf{Q}_{m}}+\varphi_{\sigma,\textbf{Q}_{m}})+f(z_{\textbf{Q}_{m}}+\varphi_{\sigma,\textbf{Q}_{m}}) \vspace{0.2cm}\\
\ds=\sum_{j=1}^{m}\sum_{k=1}^{N+1}c_{j,k}D_{j,k},~in~\mathbb{R}^{N},\vspace{0.2cm}\\
\ds Re\int\varphi_{\sigma,\textbf{Q}_{m}}\bar{D}_{j,k}=0~for~j=1,\ldots,m,k=1,\ldots,N+1.
\end{array}
\right.
\end{eqnarray}

We come to the main result in this section.

\begin{prop}\label{prop2.3}
Given $\gamma\in(0, 1)$. There exist positive numbers $\mu_{0},C ~and ~\eta > 0$ such that for all $\mu > \mu_{0}$, and for any $(\sigma,\textbf{Q}_{m})\in
[0, 2\pi]\times \Omega_{m}$ and $\epsilon< e^{-2\mu},$ there is a unique solution $(\varphi_{\sigma,\textbf{Q}_{m}}, {c_{j,k}})$ to problem \eqref{2.27}. Furthermore,
$\varphi_{\sigma,\textbf{Q}_{m}}$is $C^{1}$ in $(\sigma,\textbf{Q}_{m})$ and we have
\begin{equation}\label{2.28}
\|\varphi_{\sigma,\textbf{Q}_{m}}\|_{*}\leq Ce^{-\beta\mu},~|c_{j,k}|\leq Ce^{-\beta\mu}.
\end{equation}
\end{prop}

Note that the first equation in \eqref{2.27} can be rewritten as
\begin{equation}\label{2.29}
L(\varphi_{\sigma,\textbf{Q}_{m}})=-\mathcal {S}(z_{\textbf{Q}_{m}})+\mathcal {N}(\varphi_{\sigma,\textbf{Q}_{m}})
+\sum_{j=1}^{m}\sum_{k=1}^{N+1}c_{j,k}D_{j,k},
\end{equation}
where
\begin{equation}\label{2.30}
L(\varphi_{\sigma,\textbf{Q}_{m}})=-\Bigl(\frac{\nabla}{i}-A_{\epsilon}(x)\Bigl)^{2}\varphi_{\sigma,\textbf{Q}_{m}}
-V_{\epsilon}(x)\varphi_{\sigma,\textbf{Q}_{m}}+f'(z_{\textbf{Q}_{m}})\varphi_{\sigma,\textbf{Q}_{m}}
\end{equation}

and
\begin{equation}\label{2.31}
\mathcal
{N}(\varphi_{\sigma,\textbf{Q}_{m}})
=-\big[f(z_{\textbf{Q}_{m}}+\varphi_{\sigma,\textbf{Q}_{m}})-f(z_{\textbf{Q}_{m}})-f'(z_{\textbf{Q}_{m}})\varphi_{\sigma,\textbf{Q}_{m}}\big].
\end{equation}

In order to use the contraction mapping theorem to prove that \eqref{2.29} is uniquely solvable in the set that $\|\varphi_{\sigma,\textbf{Q}_{m}}\|_{*}$ is small, we
need to estimate $\|\mathcal {S}(z_{\textbf{Q}_{m}})\|_{*}$ and $\|\mathcal {N}(\varphi_{\sigma,\textbf{Q}_{m}})\|_{*}$ respectively.

\begin{lem}\label{lem2.4}
Given $\gamma\in(0, 1)$. For $\mu$ large enough, and any $(\sigma,\textbf{Q}_{m})\in [0,2\pi]\times \Omega_{m},$ $\epsilon< e^{-2\mu}$, we have
\begin{equation}\label{2.32}
\|\mathcal {S}(z_{\textbf{Q}_{m}})\|_{*}\leq Ce^{-\beta\mu},
\end{equation}
for some constant $\beta > 0$ and $C$ independent of $\mu,m$, $\textbf{Q}_{m}$ and $\sigma$.
\end{lem}

\begin{proof}
Note that
\begin{equation}\label{2.33}
\begin{array}{ll}
\mathcal {S}(z_{\textbf{Q}_{m}})&=\ds-\Bigl(\frac{\nabla}{i}-A_{\epsilon}(x)\Bigl)^{2}z_{\textbf{Q}_{m}}-V_{\epsilon}(x)z_{\textbf{Q}_{m}}+f(z_{\textbf{Q}_{m}})\vspace{0.2cm}\\
&=\ds-\Bigl(\frac{\nabla}{i}-A_{0}\Bigl)^{2}z_{\textbf{Q}_{m}}-z_{\textbf{Q}_{m}}-\epsilon \tilde{V}(x)z_{\textbf{Q}_{m}}+f(z_{\textbf{Q}_{m}})\vspace{0.2cm}\\
&\,\,\,\,\,\,+\ds\frac{\epsilon}{i}div\tilde{A}(x)z_{\textbf{Q}_{m}}+\frac{2\epsilon}{i}\tilde{A}(x)\cdot\nabla z_{\textbf{Q}_{m}}-2\epsilon A_{0}\cdot \tilde{A} z_{\textbf{Q}_{m}}-\epsilon^{2}|\tilde{A}(x)|^{2}z_{\textbf{Q}_{m}}\vspace{0.2cm}\\
&=\ds-\epsilon \tilde{V}(x)z_{\textbf{Q}_{m}}+f(z_{\textbf{Q}_{m}})-\sum_{j=1}^{m}f(z_{Q_{j}})\vspace{0.2cm}\\
&\,\,\,\,\,\,+\ds\frac{\epsilon}{i}div\tilde{A}(x)z_{\textbf{Q}_{m}}+\frac{2\epsilon}{i}\tilde{A}(x)\cdot\nabla z_{\textbf{Q}_{m}}-2\epsilon A_{0}\cdot \tilde{A}
z_{\textbf{Q}_{m}}-\epsilon^{2}|\tilde{A}(x)|^{2}z_{\textbf{Q}_{m}}\vspace{0.2cm}\\
&=\ds-\epsilon \tilde{V}(x)z_{\textbf{Q}_{m}}+f(z_{\textbf{Q}_{m}})-\sum_{j=1}^{m}f(z_{Q_{j}})\vspace{0.2cm}\\
&\,\,\,\,\,\,+\ds\frac{\epsilon}{i}div\tilde{A}(x)z_{\textbf{Q}_{m}}+\frac{2\epsilon}{i}\sum_{j=1}^{m}\xi_{j}\tilde{A}(x)\cdot\nabla w_{Q_{j}}-\epsilon^{2}|\tilde{A}(x)|^{2}z_{\textbf{Q}_{m}}.
\end{array}
\end{equation}

By (2.5) and (2.6) of section 2.1 in \cite{aw}, it follows from that
\begin{equation}\label{2.34}
\Bigl|f(z_{\textbf{Q}_{m}})-\sum_{j=1}^{m}f(z_{Q_{j}})\Bigl|=|e^{i\sigma}|\Bigl|f(w_{\textbf{Q}_{m}})-\sum_{j=1}^{m}f(w_{Q_{j}})\Bigl|\leq
Ce^{-\beta\mu}\sum_{j=1}^{m}e^{-\gamma|x-Q_{j}|}
\end{equation}
for a proper choice of $\beta > 0$.

Moreover, by the assumption of $\epsilon$, we can prove that
\begin{equation}\label{2.35}
|\epsilon \tilde{V}(x)z_{\textbf{Q}_{m}}|\leq Ce^{-\beta\mu}\sum_{j=1}^{m}e^{-\gamma|x-Q_{j}|}
\end{equation}
for some $\beta> 0$. In fact, on one hand, fix $j\in\{1, 2, \ldots,m\}$ and consider the region $|x -Q_{j}| \leq \frac{\mu}{2} $. In this region, we have
$$|\epsilon \tilde{V}(x)z_{\textbf{Q}_{m}}|\leq Ce^{-2\mu}\leq Ce^{-\mu}e^{-2|x -Q_{j}|}\leq Ce^{-\beta\mu}\sum_{j}e^{-\gamma|x-Q_{j}|}.$$
On the other hand, considering the region $|x -Q_{j} | > \frac{\mu}{2} $ for all $j$, we have $$|\epsilon \tilde{V}(x)z_{\textbf{Q}_{m}}|\leq
Ce^{-2\mu}|w_{\textbf{Q}_{m}}|\leq Ce^{-\beta\mu}\sum_{j=1}^{m}e^{-\gamma|x-Q_{j}|}.$$

By the same arguments with \eqref{2.35}, we can prove
\begin{equation}\label{2.36}
\bigl|\frac{\epsilon}{i}div\tilde{A}(x)z_{\textbf{Q}_{m}}\bigl|\leq Ce^{-\beta\mu}\sum_{j=1}^{m}e^{-\gamma|x-Q_{j}|},
\end{equation}
\begin{equation}\label{2.37}
\Bigl|\frac{2\epsilon}{i}\sum_{j=1}^{m}\xi_{j}\tilde{A}(x)\cdot\nabla w_{Q_{j}}\Bigl|\leq Ce^{-\beta\mu}\sum_{j=1}^{m}e^{-\gamma|x-Q_{j}|}
\end{equation}
and
\begin{equation}\label{2.38}
\Bigl|\epsilon^{2}|\tilde{A}(x)|^{2}z_{\textbf{Q}_{m}}\Bigl|\leq Ce^{-\beta\mu}\sum_{j=1}^{m}e^{-\gamma|x-Q_{j}|}.
\end{equation}

It follows from \eqref{2.33} to \eqref{2.38} that $$\|\mathcal {S}(z_{\textbf{Q}_{m}})\|_{*}\leq Ce^{-\beta\mu}$$ for some $\beta > 0$ independent of $\mu, m$ and $
\textbf{Q}_{m}$.
\end{proof}

\begin{lem}\label{lem2.5}
For any $\textbf{Q}_{m}\in \Omega_{m}$ satisfying $\|\varphi_{\sigma,\textbf{Q}_{m}}\|_{*}\leq 1$, we have
\begin{equation}\label{2.39}
\|\mathcal {N}(\varphi_{\sigma,\textbf{Q}_{m}})\|_{*}\leq C\|\varphi_{\sigma,\textbf{Q}_{m}}\|_{*}^{1+\delta}
\end{equation}
and
\begin{equation}\label{2.40}
\|\mathcal {N}(\varphi^{1}_{\sigma,\textbf{Q}_{m}})-\mathcal
{N}(\varphi^{2}_{\sigma,\textbf{Q}_{m}})\|_{*}\leq
C(\|\varphi^{1}_{\sigma,\textbf{Q}_{m}}\|_{*}^{\delta}+\|\varphi^{2}_{\sigma,\textbf{Q}_{m}}\|_{*}^{\delta})\|\varphi^{1}_{\sigma,\textbf{Q}_{m}}-\varphi^{2}_{\sigma,\textbf{Q}_{m}}\|_{*}.
\end{equation}
\end{lem}

\begin{proof}
By direct computation and applying the mean-value theorem, we have
\begin{equation}\label{2.41}
\begin{array}{ll}
|\mathcal {N}(\varphi_{\sigma,\textbf{Q}_{m}})|&
=\ds\bigl|f(z_{\textbf{Q}_{m}}+\varphi_{\sigma,\textbf{Q}_{m}})-f(z_{\textbf{Q}_{m}})-f'(z_{\textbf{Q}_{m}})\varphi_{\sigma,\textbf{Q}_{m}}\bigl|\vspace{0.2cm}\\
&=\ds\bigl|f'(z_{\textbf{Q}_{m}}+\vartheta\varphi_{\sigma,\textbf{Q}_{m}})\varphi_{\sigma,\textbf{Q}_{m}}-f'(z_{\textbf{Q}_{m}})\varphi_{\sigma,\textbf{Q}_{m}}\bigl|\vspace{0.2cm}\\
&\leq\ds C|\varphi_{\sigma,\textbf{Q}_{m}}|^{1+\delta}\leq
C\|\varphi_{\sigma,\textbf{Q}_{m}}\|_{*}^{1+\delta}\Bigl(\sum_{j=1}^{m} e^{-\gamma|x-Q_{j}|}\Bigl)^{1+\delta}\vspace{0.2cm}\\
&\ds\leq C\|\varphi_{\sigma,\textbf{Q}_{m}}\|_{*}^{1+\delta}\sum_{j=1}^{m} e^{-\gamma|x-Q_{j}|}
\end{array}
\end{equation}
and
\begin{equation}\label{2.42}
\begin{array}{ll}
&|\mathcal {N}(\varphi^{1}_{\sigma,\textbf{Q}_{m}})-\mathcal
{N}(\varphi^{2}_{\sigma,\textbf{Q}_{m}})|\vspace{0.2cm}\\
&=\ds\bigl|f(z_{\textbf{Q}_{m}}+\varphi^{1}_{\sigma,\textbf{Q}_{m}})-f(z_{\textbf{Q}_{m}}+\varphi^{2}_{\sigma,\textbf{Q}_{m}})
-f'(z_{\textbf{Q}_{m}})\varphi^{1}_{\sigma,\textbf{Q}_{m}}+f'(z_{\textbf{Q}_{m}})\varphi^{2}_{\sigma,\textbf{Q}_{m}}\bigl|\vspace{0.2cm}\\
&\ds
=\bigl|f'(z_{\textbf{Q}_{m}}+\vartheta(\varphi^{1}_{\sigma,\textbf{Q}_{m}}-\varphi^{2}_{\sigma,\textbf{Q}_{m}}))(\varphi^{1}_{\sigma,\textbf{Q}_{m}}-\varphi^{2}_{\sigma,\textbf{Q}_{m}})
-f'(z_{\textbf{Q}_{m}})(\varphi^{1}_{\sigma,\textbf{Q}_{m}}-\varphi^{2}_{\sigma,\textbf{Q}_{m}})\bigl|\vspace{0.2cm}\\
&\ds\leq C(|\varphi^{1}_{\sigma,\textbf{Q}_{m}}|^{\delta}+|\varphi^{2}_{\sigma,\textbf{Q}_{m}}|^{\delta})|\varphi^{1}_{\sigma,\textbf{Q}_{m}}-\varphi^{2}_{\sigma,\textbf{Q}_{m}}|\vspace{0.2cm}\\
&\ds\leq
C(\|\varphi^{1}_{\sigma,\textbf{Q}_{m}}\|_{*}^{\delta}+\|\varphi^{2}_{\sigma,\textbf{Q}_{m}}\|_{*}^{\delta})\|\varphi^{1}_{\sigma,\textbf{Q}_{m}}-\varphi^{2}_{\sigma,\textbf{Q}_{m}}\|_{*}
\Big(\sum_{j=1}^{m} e^{-\gamma|x-Q_{j}|}\Big)^{1+\delta}\vspace{0.2cm}\\
&\ds\leq
C(\|\varphi^{1}_{\sigma,\textbf{Q}_{m}}\|_{*}^{\delta}+\|\varphi^{2}_{\sigma,\textbf{Q}_{m}}\|_{*}^{\delta})\|\varphi^{1}_{\sigma,\textbf{Q}_{m}}-\varphi^{2}_{\sigma,\textbf{Q}_{m}}\|_{*}
\Big(\sum_{j=1}^{m} e^{-\gamma|x-Q_{j}|}\Big).
\end{array}
\end{equation}

From \eqref{2.41} and \eqref{2.42}, we can have
$$
\|\mathcal {N}(\varphi_{\sigma,\textbf{Q}_{m}})\|_{*}\leq C\|\varphi_{\sigma,\textbf{Q}_{m}}\|_{*}^{1+\delta}
$$
and
$$\|\mathcal {N}(\varphi^{1}_{\sigma,\textbf{Q}_{m}})-\mathcal
{N}(\varphi^{2}_{\sigma,\textbf{Q}_{m}})\|_{*}\leq
C(\|\varphi^{1}_{\sigma,\textbf{Q}_{m}}\|_{*}^{\delta}+\|\varphi^{2}_{\sigma,\textbf{Q}_{m}}\|_{*}^{\delta})
\|\varphi^{1}_{\sigma,\textbf{Q}_{m}}-\varphi^{2}_{\sigma,\textbf{Q}_{m}}\|_{*}.
$$
\end{proof}

Now, we are ready to prove Proposition \ref{prop2.3}.

\begin{proof}[\textbf{Proof of Proposition \ref{prop2.3}.}]
We will use the contraction theorem to prove it. Observe that $\varphi_{\sigma,\textbf{Q}_{m}}$ solves \eqref{2.27} if and only if
\begin{equation}\label{2.43}
\varphi_{\sigma,\textbf{Q}_{m}}= \mathcal {A}(-\mathcal {S}(z_{\textbf{Q}_{m}}) + \mathcal {N}(\varphi_{\sigma,\textbf{Q}_{m}}))
\end{equation}
where $\mathcal {A}$ is the operator introduced in \eqref{2.25}. In other words, $\varphi_{\sigma,\textbf{Q}_{m}}$ solves \eqref{2.27} if and only if $\varphi_{\sigma,\textbf{Q}_{m}}$
is a fixed point for the operator
$$\mathcal {T}(\varphi_{\sigma,\textbf{Q}_{m}}):= \mathcal {A}(-\mathcal {S}(z_{\textbf{Q}_{m}}) + \mathcal
{N}(\varphi_{\sigma,\textbf{Q}_{m}})).
$$

Define
$$\mathcal {B}=\Bigl\{\varphi_{\sigma,\textbf{Q}_{m}}\in H^{1}(\mathbb{R}^{N},\mathbb{C}): \|\varphi_{\sigma,\textbf{Q}_{m}}\|_{*}\leq e^{-(\beta-\tau)\mu}, ~Re\int\varphi_{\sigma,\textbf{Q}_{m}}\bar{D}_{j,k}=0\Bigl\},$$
where $\tau > 0$ small enough. We will prove that $\mathcal {T}$ is a contraction mapping from $\mathcal {B}$ to itself. On one hand, for any
$\varphi_{\sigma,\textbf{Q}_{m}}\in\mathcal {B}$, it follows from Lemmas \ref{lem2.4} and \ref{lem2.5} that
\begin{eqnarray*}
&&\|\mathcal {T}(\varphi_{\sigma,\textbf{Q}_{m}})\|_{*}\leq C\|-\mathcal {S}(z_{\textbf{Q}_{m}}) + \mathcal {N}(\varphi_{\sigma,\textbf{Q}_{m}})\|_{*}\\
&\leq &Ce^{-\beta\mu}+C\|\varphi_{\sigma,\textbf{Q}_{m}}\|_{*}^{1+\delta}
\leq Ce^{-\beta\mu}+Ce^{-(1+\delta)(\beta-\tau)\mu}\leq e^{-(\beta-\tau)\mu}.
\end{eqnarray*}

On the other hand, taking $\varphi_{\sigma,\textbf{Q}_{m}}^{1}$ and $\varphi_{\sigma,\textbf{Q}_{m}}^{2}$ in $\mathcal {B}$, by Lemma \ref{lem2.5} we have
\begin{eqnarray*}
&&\|\mathcal {T}(\varphi_{\sigma,\textbf{Q}_{m}}^{1})
-\mathcal {T}(\varphi_{\sigma,\textbf{Q}_{m}}^{1})\|_{*}\leq C\|\mathcal {N}(\varphi_{\sigma,\textbf{Q}_{m}}^{1}) - \mathcal {N}(\varphi_{\sigma,\textbf{Q}_{m}}^{2})\|_{*}
\\&\leq& C(\|\varphi_{\sigma,\textbf{Q}_{m}}^{1}\|_{*}^{\delta}+\|\varphi_{\sigma,\textbf{Q}_{m}}^{2}\|_{*}^{\delta})\|\varphi_{\sigma,\textbf{Q}_{m}}^{1} - \varphi_{\sigma,\textbf{Q}_{m}}^{2}\|_{*}
\leq \frac{1}{2}\|\varphi_{\sigma,\textbf{Q}_{m}}^{1} - \varphi_{\sigma,\textbf{Q}_{m}}^{2}\|_{*}.
\end{eqnarray*}

Hence by the contraction mapping theorem, for any $(\sigma, \textbf{Q}_{m})\in [0,2\pi]\times \Omega_{m},$
there exists a unique $\varphi_{\sigma,\textbf{Q}_{m}}\in\mathcal {B}$ such that \eqref{2.43} holds. So $$\|
\varphi_{\sigma,\textbf{Q}_{m}}\|_{*}=\|\mathcal {T}(\varphi_{\sigma,\textbf{Q}_{m}})\|_{*}\leq Ce^{-\beta\mu}.$$

Now we need to prove that $\varphi_{\sigma,\textbf{Q}_{m}}$ is $2\pi$-periodic with respect to $\sigma.$
Replacing $\sigma$ by $\sigma+2\pi$ in the above reduction process, we get $\varphi_{\sigma+2\pi,\textbf{Q}_{m}}.$
Since $z_{\textbf{Q}_{m}}$ is $2\pi$-periodic, by the uniqueness
of $\varphi_{\sigma,\textbf{Q}_{m}}$, we see
$\varphi_{\sigma,\textbf{Q}_{m}}=\varphi_{\sigma+2\pi,\textbf{Q}_{m}}.$

Combining \eqref{2.20}, \eqref{2.32}, \eqref{2.39} and \eqref{2.40} we have $$|c_{j,k}|\leq C(e^{-\frac{\beta\mu}{2}}\|\varphi_{\sigma,\textbf{Q}_{m}}\|_{*}+\|\mathcal
{S}(\varphi_{\sigma,\textbf{Q}_{m}})\|_{*}+\|\mathcal {N}(\varphi_{\sigma,\textbf{Q}_{m}})\|_{*})\leq Ce^{-\beta\mu}.$$
\end{proof}

\section{A secondary Lyapunov-Schmidt reduction}

In this section, we present a key estimate on the difference between the solutions in the m-th step and (m+1)-th step. This second Lyapunov-Schmidt reduction
has been used in the paper \cite{aw,lw,w}. For $(\sigma,\textbf{Q}_{m})\in [0,2\pi]\times\Omega_{m}$, we denote $u_{\textbf{Q}_{m}}$ as $z_{\textbf{Q}_{m}}+\varphi_{\sigma,\textbf{Q}_{m}}$,
where $\varphi_{\sigma,\textbf{Q}_{m}}$ is the unique solution given by Proposition \ref{prop2.3}. The main estimate below states that the difference between
$u_{\textbf{Q}_{m+1}}$ and $u_{\textbf{Q}_{m}}+z_{Q_{m+1}}$ is small globally in $H^{1}(\mathbb{R}^{N},\mathbb{C})$ norm.

For this purpose, we now write
\begin{equation}\label{3.1}
u_{\textbf{Q}_{m+1}}=u_{\textbf{Q}_{m}}+z_{Q_{m+1}}+\phi_{m+1}=:\bar{u}+\phi_{m+1}.
\end{equation}

By Proposition \ref{prop2.3}, we can easily obtain that
\begin{equation}\label{3.2}
\|\phi_{m+1}\|_{*}\leq Ce^{-\beta\mu}.
\end{equation}
However the estimate \eqref{3.2} is not sufficient. We need a crucial estimate for $\phi_{m+1}$ which will be given later. (In the following we will always assume
that $\gamma > \frac{1}{2}$) In order to obtain the crucial estimate, we will need the following lemma.

\begin{lem}\label{lem3.1}
(Lemma 2.3, \cite{bl}) For $|Q_{j}-Q_{k}|\geq\mu$ large, it holds that
\begin{equation}\label{3.3}
\int f(w(x-Q_{j}))w(x-Q_{k})dx=(\vartheta+e^{-\beta\mu})w(|Q_{j}-Q_{k}|)
\end{equation}
 for some $\beta > 0$ independent of large $\mu$ and
\begin{equation}\label{3.4}
\vartheta=\int f(w)e^{-x_{1}}dx >0.
\end{equation}
\end{lem}

\begin{lem}\label{lem3.2}
Let $\mu$, $\epsilon$ be as in Proposition \ref{prop2.3}. Then it holds
\begin{equation}\label{3.5}
\begin{array}{ll}
&\|\phi_{m+1}\|_{H^{1}(\mathbb{R}^{N})}\vspace{0.2cm}\\
&
\leq \ds C\Bigl[\epsilon\int|\tilde{V}(x)||w_{Q_{m+1}}|+2\epsilon\int|\tilde{A}(x)||\nabla w_{Q_{m+1}}|+\epsilon\int|divA(x)||w_{Q_{m+1}}|
\vspace{0.2cm}\\
&\quad\quad\ds +\epsilon^{2}\int|\tilde{A}(x)|^{2}|w_{Q_{m+1}}|+e^{-\beta\mu}\bigl(\sum_{j=1}^{m}w(|Q_{m+1}-Q_{j}|)\bigl)^{\frac{1}{2}}
+\epsilon\bigl(\int|\tilde{V}(x)|^{2}|w_{Q_{m+1}}|^{2}\bigl)^{\frac{1}{2}}\vspace{0.2cm}\\
&\quad\quad
+\epsilon\bigl(\ds\int|\tilde{A}(x)|^{2}|\nabla
w_{Q_{m+1}}|^{2}\bigl)^{\frac{1}{2}}+\epsilon\bigl(\ds\int|divA(x)|^{2}|w_{Q_{m+1}}|^{2}\bigl)^{\frac{1}{2}}
+\epsilon^{2}\bigl(\ds\int|\tilde{A}(x)|^{4}|w_{Q_{m+1}}|^{2}\bigl)^{\frac{1}{2}}\Bigl]
\end{array}
\end{equation}
for some constant $C > 0$, $ \beta> 0$ independent of $\mu,$ $m,$ $\gamma$ and $\textbf{Q}_{m+1} \in\Omega_{m+1}$.
\end{lem}

\begin{proof}
To prove \eqref{3.5}, we need to perform a further decomposition.

As we mentioned before, the following eigenvalue problem
$$\Delta
\varphi -\varphi + f'(w)\varphi = \lambda\varphi, ~~\varphi \in H^{1}(\mathbb{R}^{N}),$$
admits the following set of eigenvalues
$$
\lambda_{1}>\lambda_{2}>\ldots>\lambda_{n}>\lambda_{n+1}=0>\lambda_{n+2}\ldots.
$$
We denote the eigenfunctions corresponding to the positive eigenvalues $\lambda_{j}$ as $\varphi_{j},$ $j=1,\ldots,n.$

Now, we have the eigenvalue $\lambda_{k}(k=1,\ldots,n)$ with eigenfunction
$\tilde{\varphi}_{0,k}=e^{i\sigma+iA_{0}\cdot x}\varphi_{k}$ of the following linearized operator
\begin{equation}\label{3.6}
-\Bigl(\frac{\nabla}{i}-A_{0}\Bigl)^{2}\varphi-\varphi+f'(w)\varphi=\lambda\varphi.
\end{equation}
We fix $\tilde{\varphi}_{0,k}$ such that $\max_{x\in\mathbb{R}^{N}} |\tilde{\varphi}_{0,k}| = 1$. Denote by $\tilde{\varphi}_{j,k} =\eta_{j}\tilde{\varphi}_{0,k}(x -Q_{j})$, where $\eta_{j}$ is the cut-off function introduced in section 1.

By the equations satisfied by $\phi_{m+1}$, we have
\begin{equation}\label{3.7}
\bar{L}\phi_{m+1}=-\bar{\mathcal {S}}+\sum_{j=1}^{m+1}\sum_{k=1}^{N+1}c_{j,k}D_{j,k}
\end{equation}
for some constants ${c_{j,k}}$, where
$$\bar{L}=-\Bigl(\frac{\nabla}{i}-A_{\epsilon}(x)\Bigl)^{2}-V_{\epsilon}(x)+f'(\tilde{u}),$$

where
\begin{eqnarray*}f'(\tilde{u})=
\left\{%
\begin{array}{ll}
   \ds \frac{f(\bar{u}+\phi_{m+1})-f(\bar{u})}{\phi_{m+1}},~~&\text{if}~~\phi_{m+1}\neq0, \vspace{0.2cm}\\
\ds f'(\bar{u}),~~~~~~~~~~~~~~~~~~~~~~&\text{if}~~\phi_{m+1}=0,
\end{array}
\right.
\end{eqnarray*}

and
\begin{equation}\label{3.8}
\begin{array}{ll}
\bar{\mathcal {S}}&=\ds f(u_{\textbf{Q}_{m}}-z_{Q_{m+1}})-f(u_{\textbf{Q}_{m}})-(1+\epsilon \tilde{V}(x))z_{Q_{m+1}}-\Big(\frac{\nabla}{i}-A_{\epsilon}(x)\Big)^{2}z_{Q_{m+1}}\vspace{0.2cm}\\
&=\ds f(u_{\textbf{Q}_{m}}-z_{Q_{m+1}})-f(u_{\textbf{Q}_{m}})-f(z_{Q_{m+1}})-\epsilon
\tilde{V}(x)z_{Q_{m+1}}+\frac{\epsilon}{i}div\tilde{A}(x)z_{Q_{m+1}}\vspace{0.2cm}\\
&\,\,\,\,\,\,\ds+2\frac{\epsilon}{i}\tilde{A}(x)\cdot\nabla z_{Q_{m+1}}-2\epsilon A_{0}\cdot \tilde{A}z_{Q_{m+1}}-\epsilon^{2}|\tilde{A}(x)|^{2}z_{Q_{m+1}}\vspace{0.2cm}\\
&=\ds f(u_{\textbf{Q}_{m}}-z_{Q_{m+1}})-f(u_{\textbf{Q}_{m}})-f(z_{Q_{m+1}})-\epsilon
\tilde{V}(x)z_{Q_{m+1}}+\frac{\epsilon}{i}div\tilde{A}(x)z_{Q_{m+1}}\vspace{0.2cm}\\
&\,\,\,\,\,\,\ds+2\frac{\epsilon}{i}\xi_{j}\tilde{A}(x)\cdot\nabla w_{Q_{m+1}}-\epsilon^{2}|\tilde{A}(x)|^{2}z_{Q_{m+1}}.
\end{array}
\end{equation}

Now we proceed the proof into a few steps.

First we estimate the $L^{2}$-norm of $\bar{\mathcal {S}}$. By the estimate in Proposition \ref{prop2.3}, we have the following estimate
\begin{equation}\label{3.9}
\int|f(u_{\textbf{Q}_{m}}+z_{Q_{m+1}})-f(u_{\textbf{Q}_{m}})-f(z_{Q_{m+1}})|^{2}\leq Ce^{-\beta\mu}\sum_{j=1}^{m}w(|Q_{m+1}-Q_{j}|).
\end{equation}

We also have
$$\int|\epsilon \tilde{V}(x)z_{Q_{m+1}}|^{2}\leq C\epsilon^{2}\int \tilde{V}(x)^{2}w_{Q_{m+1}}^{2},$$
$$\int  \Big|\frac{\epsilon}{i} div\tilde{A} z_{Q_{m+1}} \Big|^{2}\leq C\epsilon^{2}\int |div\tilde{A}|^{2}w_{Q_{m+1}}^{2},$$
$$\int  \Big|2\frac{\epsilon}{i}\xi_{j} \tilde{A}(x)\cdot\nabla w_{Q_{m+1}} \Big|^{2}\leq C\epsilon^{2}\int |\tilde{A}(x)|^{2}|\nabla w_{Q_{m+1}}|^{2}$$
and
\begin{equation}\label{3.10}
\int \bigl|\epsilon^{2}|\tilde{A}(x)|^{2}z_{Q_{m+1}}\bigl|^{2}\leq C\epsilon^{4}\int |\tilde{A}(x)|^{4}w_{Q_{m+1}}^{2}.
\end{equation}

It follows from \eqref{3.8} to \eqref{3.10} that
\begin{equation}\label{3.11}
\begin{array}{ll}
\|\bar{\mathcal {S}}\|^{2}_{L^{2}}&\leq \ds Ce^{-\beta\mu}\sum_{j=1}^{m}w(|Q_{m+1}-Q_{j}|)+C\epsilon^{2}\int \tilde{V}(x)^{2}w_{Q_{m+1}}^{2}+C\epsilon^{2}\int
|div\tilde{A}|^{2}w_{Q_{m+1}}^{2}\vspace{0.2cm}\\
&\,\,\,\,\,\,\ds+C\epsilon^{2}\int |\tilde{A}(x)|^{2}|\nabla w_{Q_{m+1}}|^{2}+C\epsilon^{4}\int |\tilde{A}(x)|^{4}w_{Q_{m+1}}^{2}.
\end{array}
\end{equation}

By the estimate \eqref{3.2}, we have the following estimate
\begin{equation}\label{3.12}
\Bigl|\tilde{u}-\sum_{j=1}^{m+1}z(x-Q_{j})\Bigl|=O(e^{-\beta\mu}).
\end{equation}

Decompose $\phi_{m+1}$ as
\begin{equation}\label{3.13}
\phi_{m+1}=\psi+\sum_{j=1}^{m+1}\sum_{l=1}^{n}g_{j,l}\tilde{\varphi}_{j,l}+\sum_{j=1}^{m+1}\sum_{k=1}^{N+1}d_{j,k}D_{j,k}
\end{equation}
for some $g_{j,l}$ , $d_{j,k}$ such that
\begin{equation}\label{3.14}
Re\int\psi\bar{\tilde{\varphi}}_{j,l}=Re\int\psi \bar{D}_{j,k}=0,j=1,\ldots,m+1,k=1,\ldots,N+1,l=1,\ldots,n.
\end{equation}

Since
\begin{equation}\label{3.15}
\phi_{m+1}=\varphi_{\sigma,\textbf{Q}_{m+1}}-\varphi_{\sigma,\textbf{Q}_{m}},
\end{equation}
we have for $j = 1, \ldots, m,$
\begin{equation}\label{3.16}
d_{j,k}=Re\int\phi_{m+1}\bar{D}_{j,k}=Re\int(\varphi_{\sigma,\textbf{Q}_{m+1}}-\varphi_{\sigma,\textbf{Q}_{m}})\bar{D}_{j,k}=0
\end{equation}
and
\begin{equation}\label{3.17}
d_{m+1,k}=Re\int\phi_{m+1}\bar{D}_{m+1,k}=Re\int(\varphi_{\sigma,\textbf{Q}_{m+1}}
-\varphi_{\sigma,\textbf{Q}_{m}})\bar{D}_{m+1,k}=-Re\int\varphi_{\sigma,\textbf{Q}_{m}}\bar{D}_{m+1,k},
\end{equation}
where we use the orthogonality conditions satisfied by $\varphi_{\sigma,\textbf{Q}_{m}}$ and $\varphi_{\sigma,\textbf{Q}_{m+1}}$. Hence by Proposition \ref{prop2.3}, we have
\begin{equation}\label{3.18}
   \ds d_{j,k}=0,~for~ j = 1, \ldots, m,\,\,\,\text{and}\,\,\,
\ds |d_{m+1,k}|\leq Ce^{-\beta\mu}\sum_{j=1}^{m}e^{-\gamma|Q_{j}-Q_{m+1}|}.
\end{equation}
By \eqref{3.13}, we can rewrite \eqref{3.7} as
\begin{equation}\label{3.19}
\begin{array}{ll}
\bar{L}(\psi)+\ds\sum_{j=1}^{m+1}\sum_{l=1}^{n}g_{j,l}\bar{L}(\tilde{\varphi}_{j,l})+\ds\sum_{j=1}^{m+1}\sum_{k=1}^{N+1}d_{j,k}\bar{L}(D_{j,k})
=-\bar{\mathcal {S}}+\ds\sum_{j=1}^{m+1}\sum_{k=1}^{N+1}c_{j,k}D_{j,k}.
\end{array}
\end{equation}

In order to estimate the coefficients $g_{j,l}$, we use the equation \eqref{3.19}. First, multiplying \eqref{3.19} by $\tilde{\varphi}_{j,l}$ and integrating over $\mathbb{R}^{N}$,
we have
\begin{equation}\label{3.20}
\begin{array}{ll}
\ds~Re~g_{j,l}\int\bar{L}(\tilde{\varphi}_{j,l})\bar{\tilde{\varphi}}_{j,l}&=-\ds\sum_{j=1}^{m+1}\sum_{k=1}^{N+1}Re~d_{j,k}\int\bar{L}(D_{j,k})\bar{\tilde{\varphi}}_{j,l}
-\sum_{k\neq l}Re~g_{j,k}\int\bar{L}(\tilde{\varphi}_{j,k})\bar{\tilde{\varphi}}_{j,l}
-Re\int\bar{\mathcal {S}}\bar{\tilde{\varphi}}_{j,l}\vspace{0.2cm}\\
&\,\,\,\,\,\,\ds-Re\int\bar{L}(\psi)\bar{\tilde{\varphi}}_{j,l},
\end{array}
\end{equation}
where
\begin{eqnarray}\label{3.21}
\left\{%
\begin{array}{ll}
   \ds \Bigl|Re\int\bar{\mathcal {S}}\bar{\tilde{\varphi}}_{j,l}\Bigl|\leq Ce^{-\beta\mu}e^{-\gamma|Q_{j}-Q_{m+1}|}+\epsilon\Bigl|Re\int \tilde{V}(x)z_{Q_{m+1}}\bar{\tilde{\varphi}}_{j,l}\Bigl|+2\epsilon\Bigl|Re\int \frac{1}{i}\xi_{j}\tilde{A}(x)\cdot\nabla w_{Q_{m+1}}\bar{\tilde{\varphi}}_{j,l}\Bigl| \vspace{0.2cm}\\
\quad\quad\quad\quad\quad\quad
+\epsilon\Bigl|Re\ds\int \frac{1}{i} div \tilde{A}(x)z_{Q_{m+1}}\bar{\tilde{\varphi}}_{j,l}\Bigl|+\epsilon^{2}\Bigl|Re\int |\tilde{A}(x)|^{2}z_{Q_{m+1}}\bar{\tilde{\varphi}}_{j,l}\Bigl|,~j=1,\ldots,m,\vspace{0.2cm}\\
\ds \Bigl|Re\int\bar{\mathcal {S}}\bar{\tilde{\varphi}}_{m+1,l}\Bigl|\leq Ce^{-\beta\mu}\sum_{j=1}^{m}e^{-\gamma|Q_{j}-Q_{m+1}|}+\epsilon\Bigl|Re\int \tilde{V}(x)z_{Q_{m+1}}\bar{\tilde{\varphi}}_{m+1,l}\Bigl| \vspace{0.2cm}\\
\quad\quad\quad\quad\quad\quad\quad
+2\epsilon\Bigl|Re\ds\int\frac{1}{i}\xi_{j} \tilde{A}(x)\cdot\nabla w_{Q_{m+1}}\bar{\tilde{\varphi}}_{m+1,l}\Bigl|+\epsilon\Bigl|Re\int\frac{1}{i} div
\tilde{A}(x)z_{Q_{m+1}}\bar{\tilde{\varphi}}_{m+1,l}\Bigl|\vspace{0.2cm}\\
\quad\quad\quad\quad\quad\quad\quad
+\epsilon^{2}\Bigl|Re\ds\int |\tilde{A}(x)|^{2}z_{Q_{m+1}}\bar{\tilde{\varphi}}_{m+1,l}\Bigl|.
\end{array}
\right.
\end{eqnarray}

By the definition of $\tilde{\varphi}_{j,l}$, we have
$$\bar{L}(\tilde{\varphi}_{j,k}) = \lambda_{k}\tilde{\varphi}_{j,k} + O ( e^{-\beta\mu} ),$$
thus one has
\begin{equation}\label{3.22}
Re\int\bar{L}(\tilde{\varphi}_{j,k})\bar{\tilde{\varphi}}_{j,l} = -\delta_{k,l}\lambda_{k} \int\tilde{\varphi}_{0,l}\tilde{\varphi}_{0,k} + O ( e^{-\beta\mu} ).
\end{equation}
Recall the definition of $\varphi$, we have
\begin{equation*}
\begin{array}{ll}
\ds~Re~\int\bar{L}(\psi)\bar{\tilde{\varphi}}_{j,l}&=-\ds~Re~\int\psi\bar{\bar{L}}(\tilde{\varphi}_{j,l})=\ds -\lambda_{l}\int\tilde{\varphi}_{j,l}\psi + O ( e^{-\beta\mu} )\|\psi\|_{H^{1}(B_{\frac{\mu}{2}} (Q_{j}))}\vspace{0.2cm}\\
&=\ds O ( e^{-\beta\mu} )\|\psi\|_{H^{1}(B_{\frac{\mu}{2}} (Q_{j}))}.
\end{array}
\end{equation*}

Combining \eqref{3.18}, \eqref{3.20}, \eqref{3.21} and \eqref{3.22}, and the orthogonal conditions satisfied by $\psi$
\begin{eqnarray}\label{3.23}
\left\{%
\begin{array}{ll}
   \ds |g_{j,l}|\leq Ce^{-\beta\mu}e^{-\gamma|Q_{j}-Q_{m+1}|}+\epsilon\Bigl|Re\int \tilde{V}(x)z_{Q_{m+1}}\bar{\tilde{\varphi}}_{j,l}\Bigl|+2\epsilon\Bigl|Re\int \frac{1}{i}\xi_{j}\tilde{A}(x)\cdot\nabla w_{Q_{m+1}}\bar{\tilde{\varphi}}_{j,l}\Bigl| \vspace{0.2cm}\\
\quad\quad\quad+\epsilon\Bigl|Re\ds\int \frac{1}{i} div \tilde{A}(x)z_{Q_{m+1}}\bar{\tilde{\varphi}}_{j,l}\Bigl|+\epsilon^{2}\Bigl|Re\int |\tilde{A}(x)|^{2}z_{Q_{m+1}}\bar{\tilde{\varphi}}_{j,l}\Bigl|\vspace{0.2cm}\\
\quad\quad\quad+e^{-\beta\mu}\|\psi\|_{H^{1}(B_{\frac{\mu}{2}} (Q_{j}))},~j=1,\ldots,m,\vspace{0.2cm}\\
\ds |g_{m+1,l}|\leq Ce^{-\beta\mu}\sum_{j=1}^{m}e^{-\gamma|Q_{j}-Q_{m+1}|}+\epsilon\Bigl|Re\int \tilde{V}(x)z_{Q_{m+1}}\bar{\tilde{\varphi}}_{m+1,l}\Bigl| \vspace{0.2cm}\\
\quad\quad\quad+2\epsilon\Bigl|Re\ds\int\frac{1}{i} \xi_{j}\tilde{A}(x)\cdot\nabla w_{Q_{m+1}}\bar{\tilde{\varphi}}_{m+1,l}\Bigl|+\epsilon\Bigl|Re\int\frac{1}{i} div \tilde{A}(x)z_{Q_{m+1}}\bar{\tilde{\varphi}}_{m+1,l}\Bigl|\vspace{0.2cm}\\
\quad\quad\quad+\epsilon^{2}\Bigl|Re\ds\int |\tilde{A}(x)|^{2}z_{Q_{m+1}}\bar{\tilde{\varphi}}_{m+1,l}\Bigl|+e^{-\beta\mu}\|\psi\|_{H^{1}(B_{\frac{\mu}{2}} (Q_{m+1} ))}.
\end{array}
\right.
\end{eqnarray}

Next, we estimate $\psi$. Multiplying \eqref{3.19} by $\psi$ and integrating over $\mathbb{R}^{N}$, we find
\begin{equation}\label{3.24}
\begin{array}{ll}
Re\ds\int\bar{L}(\psi)\bar{\psi}&=-Re\ds\int\bar{\mathcal {S}}\bar{\psi}-\sum_{j=1}^{m+1}\sum_{k=1}^{N+1}d_{j,k}Re\int\bar{L}(D_{j,k})\bar{\psi}
-\ds\sum_{j=1}^{m+1}\sum_{l=1}^{n}g_{j,l}Re\int\bar{L}(\varphi_{j,l})\bar{\psi}.
\end{array}
\end{equation}

We claim that
\begin{equation}\label{3.25}
Re\int(-\bar{L}(\psi)\bar{\psi})\geq c_{0}\|\psi\|^{2}_{H^{1}}
\end{equation}
for some constant $c_{0} > 0$.

Since the approximate solution is exponentially decay away from the points $Q_{j}$ , we have
\begin{equation}\label{3.26}
Re\int_{\mathbb{R}^{N}\backslash\cup_{j}B_{\frac{\mu}{2}}(Q_{j})}(-\bar{L}(\psi)\bar{\psi})\geq
\frac{1}{2}\int_{\mathbb{R}^{N}\backslash\cup_{j}B_{\frac{\mu}{2}}(Q_{j})}(|\nabla\psi|^{2}+|\psi|^{2}).
\end{equation}

Now we only need to prove the above estimates in the domain $\cup_{j}B_{\frac{\mu}{2}(Q_{j})}$. We prove it by contradiction. Otherwise, there exists a
sequence $\mu_{n}\rightarrow\infty$, and $Q^{(n)}_{j}$ such that
$$
\int_{B_{\frac{\mu_{n}}{2}}(Q_{j}^{(n)})}(|\nabla\psi_{n}|^{2}+|\psi_{n}|^{2})=1,~Re\int_{B_{\frac{\mu_{n}}{2}}(Q_{j}^{(n)})}(-\bar{L}(\psi_{n})\bar{\psi}_{n})\rightarrow0,~as ~n\rightarrow\infty.
$$
Then we can extract from the sequence $\psi_{n}(\cdot-Q_{j}^{(n)})$ a subsequence which will converge weakly in $H^{1}(\mathbb{R}^{N})$ to $\psi_{\infty}$,
and $\mu_{n}\rightarrow\infty$, we have
\begin{equation}\label{3.27}
\int\Bigl|\bigl(\frac{\nabla}{i}-A_{0}\bigl)\psi_{\infty}\Bigl|^{2}+|\psi_{\infty}|^{2}-f'(e^{i\sigma+iA_{0}\cdot x}w)\psi_{\infty}^{2}=0
\end{equation}
and
\begin{equation}\label{3.28}
Re\int\psi_{\infty}\bar{\tilde{\varphi}}_{0,l}=Re\int\psi_{\infty}\frac{\partial(\overline{e^{i\sigma+iA_{0}\cdot x}w})}{\partial x_{j}}=0,~~ j=1,\ldots,N,l=1,\ldots,n.
\end{equation}

It follows from \eqref{3.27} and \eqref{3.28} that $\psi_{\infty}= 0$. Therefore
\begin{equation}\label{3.29}
\psi_{n}\rightharpoonup0 ~weakly ~in ~H^{1}(\mathbb{R}^{N}).
\end{equation}

Hence, we have
\begin{equation}\label{3.30}
\int_{B_{\frac{\mu_{n}}{2}}(Q_{j}^{(n)})}f'(\tilde{u})\psi_{n}^{2}\rightarrow0,~as ~n\rightarrow\infty.
\end{equation}
Then $$\|\psi_{n}\|_{H^{1}(B_{\frac{\mu_{n}}{2}}(Q_{j}^{(n)}))}\rightarrow0,~as~ n\rightarrow\infty,$$ which contradicts to the assumption
$\|\psi_{n}\|_{H^{1}} = 1$. Therefore \eqref{3.25} holds.

It follows from \eqref{3.24} and \eqref{3.25} that
\begin{equation}\label{3.31}
\begin{array}{ll}
&\|\psi\|^{2}_{H^{1}(\mathbb{R}^{N})}\vspace{0.2cm}\\&\leq \ds C\Bigl(\sum_{j,k}|d_{j,k}| \Bigl|Re\int\bar{L}(D_{j,k})\bar{\psi}\Bigl|+\sum_{j,l}|g_{j,l}| \Bigl|Re\int\bar{L}(\varphi_{j,l})\bar{\psi}\Bigl|
\ds +\Bigl|Re\int\bar{\mathcal {S}}\bar{\psi}\Bigl| \Bigl)  \vspace{0.2cm}\\
&\leq \ds C\Bigl(\sum_{j,k}|d_{j,k}| \|\psi\|_{H^{1}}+\sum_{j,l}|g_{j,l}| \|\psi\|_{H^{1}(B_{\frac{\mu}{2}(Q_{j})})}+\|\bar{\mathcal
{S}}\|_{L^{2}}\|\psi\|_{H^{1}}\Bigl).
\end{array}
\end{equation}

By \eqref{3.23} and \eqref{3.31}, we have
\begin{equation}\label{3.32}
\begin{array}{ll}
\|\psi\|_{H^{1}(\mathbb{R}^{N})}&\leq \ds C\Bigl(\sum_{j,k}|d_{jk}|+e^{-\beta\mu}\sum_{j=1}^{m}e^{-\gamma|Q_{j}-Q_{m+1}|}+\|\bar{\mathcal {S}}\|_{L^{2}}+\epsilon\int|\tilde{V}(x)||w_{Q_{m+1}}|\vspace{0.2cm}\\
&\,\,\,\,\,\,\ds +2\epsilon\int|\tilde{A}(x)||\nabla w_{Q_{m+1}}|+\epsilon\int|div\tilde{A}(x)||w_{Q_{m+1}}|+\epsilon^{2}\int|\tilde{A}(x)|^{2}|w_{Q_{m+1}}|\Bigl).
\end{array}
\end{equation}

From \eqref{3.11}, \eqref{3.18} and \eqref{3.32}, recalling that $ \gamma> \frac{1}{2}$ , we get
\begin{equation}\label{3.33}
\begin{array}{ll}
&\|\phi_{m+1}\|_{H^{1}(\mathbb{R}^{N})}\\&\leq \ds C\Bigl[e^{-\beta\mu}\sum_{j=1}^{m}e^{-\gamma|Q_{j}-Q_{m+1}|}+\epsilon\int|\tilde{V}(x)||w_{Q_{m+1}}|+2\epsilon\int|\tilde{A}(x)||\nabla w_{Q_{m+1}}|\vspace{0.2cm}\\
&\quad\quad\ds +\epsilon\int|div\tilde{A}(x)||w_{Q_{m+1}}|+\epsilon^{2}\int|\tilde{A}(x)|^{2}|w_{Q_{m+1}}|\vspace{0.2cm}\\
&\quad\quad\ds +e^{-\beta\mu}\bigl(\sum_{j=1}^{m}w(|Q_{m+1}-Q_{j}|)\bigl)^{\frac{1}{2}}+\epsilon\bigl(\int|\tilde{V}(x)|^{2}|w_{Q_{m+1}}|^{2}\bigl)^{\frac{1}{2}}\vspace{0.2cm}\\
&\quad\quad\ds +\epsilon\bigl(\int|\tilde{A}(x)|^{2}|\nabla
w_{Q_{m+1}}|^{2}\bigl)^{\frac{1}{2}}+\epsilon\bigl(\int|div\tilde{A}(x)|^{2}|w_{Q_{m+1}}|^{2}\bigl)^{\frac{1}{2}}+\epsilon^{2}\bigl(\int|\tilde{A}(x)|^{4}|w_{Q_{m+1}}|^{2}\bigl)^{\frac{1}{2}}\Bigl].
\end{array}
\end{equation}

Since we choose $ \gamma> \frac{1}{2}$, by the definition of the configuration space, we have
\begin{equation}\label{3.34}
\Bigl(\sum_{j=1}^{m}e^{-\gamma|Q_{j}-Q_{m+1}|}\Bigl)^{2}\leq C\sum_{j=1}^{m}w(|Q_{m+1}-Q_{j}|).
\end{equation}
It follows from \eqref{3.33} and \eqref{3.34} that
\begin{equation}\label{3.35}
\begin{array}{ll}
&\|\phi_{m+1}\|_{H^{1}(\mathbb{R}^{N})}\\
&\leq \ds C\Bigl[\epsilon\int|\tilde{V}(x)||w_{Q_{m+1}}|+2\epsilon\int|\tilde{A}(x)||\nabla w_{Q_{m+1}}|+\epsilon\int|div\tilde{A}(x)||w_{Q_{m+1}}|\vspace{0.2cm}\\
&\quad\quad\ds +\epsilon^{2}\int|\tilde{A}(x)|^{2}|w_{Q_{m+1}}|+e^{-\beta\mu}\bigl(\sum_{j=1}^{m}w(|Q_{m+1}-Q_{j}|)\bigl)^{\frac{1}{2}}+\epsilon\bigl(\int|\tilde{V}(x)|^{2}|w_{Q_{m+1}}|^{2}\bigl)^{\frac{1}{2}}\vspace{0.2cm}\\
&\quad\quad\ds +\epsilon\bigl(\int|\tilde{A}(x)|^{2}|\nabla
w_{Q_{m+1}}|^{2}\bigl)^{\frac{1}{2}}+\epsilon\bigl(\int|div\tilde{A}(x)|^{2}|w_{Q_{m+1}}|^{2}\bigl)^{\frac{1}{2}}+\epsilon^{2}\bigl(\int|\tilde{A}(x)|^{4}|w_{Q_{m+1}}|^{2}\bigl)^{\frac{1}{2}}\Bigl].
\end{array}
\end{equation}
Hence \eqref{3.5} holds.

Moreover, from the estimates \eqref{3.18} and \eqref{3.23}, and taking into consideration that $\eta_{j}$ is supposed in $B_{\frac{\mu}{2}}(Q_{j})$, using the $H\ddot{o}lder$
inequality, we can get a more accurate estimate on $\phi_{m+1}$,
\begin{equation}\label{3.36}
\begin{array}{ll}
&\|\phi_{m+1}\|_{H^{1}(\mathbb{R}^{N})}\\
&\leq \ds C\Bigl[\epsilon\sum_{j=1}^{m+1}\bigl(\int_{B_{\frac{\mu}{2}}(Q_{j})}|\tilde{V}(x)|^{2}|w_{Q_{m+1}}|^{2}\bigl)^{\frac{1}{2}}
+2\epsilon\sum_{j=1}^{m+1}\bigl(\int_{B_{\frac{\mu}{2}}(Q_{j})}|\tilde{A}(x)|^{2}|\nabla w_{Q_{m+1}}|^{2}\bigl)^{\frac{1}{2}}\vspace{0.2cm}\\
&\quad\quad\ds+\epsilon\sum_{j=1}^{m+1}\bigl(\int_{B_{\frac{\mu}{2}}(Q_{j})}|div\tilde{A}(x)|^{2}|w_{Q_{m+1}}|^{2}\bigl)^{\frac{1}{2}} +\epsilon^{2}\sum_{j=1}^{m+1}\bigl(\int_{B_{\frac{\mu}{2}}(Q_{j})}|\tilde{A}(x)|^{4}|w_{Q_{m+1}}|^{2}\bigl)^{\frac{1}{2}}\vspace{0.2cm}\\
&\quad\quad\ds+e^{-\beta\mu}\bigl(\sum_{j=1}^{m}w(|Q_{m+1}-Q_{j}|)\bigl)^{\frac{1}{2}}+\epsilon\bigl(\int|\tilde{V}(x)|^{2}|w_{Q_{m+1}}|^{2}\bigl)^{\frac{1}{2}}+\epsilon^{2}\bigl(\int|\tilde{A}(x)|^{4}|w_{Q_{m+1}}|^{2}\bigl)^{\frac{1}{2}}\vspace{0.2cm}\\
&\quad\quad\ds +\epsilon\bigl(\int|\tilde{A}(x)|^{2}|\nabla
w_{Q_{m+1}}|^{2}\bigl)^{\frac{1}{2}}+\epsilon\bigl(\int|div\tilde{A}(x)|^{2}|w_{Q_{m+1}}|^{2}\bigl)^{\frac{1}{2}}\Bigl].
\end{array}
\end{equation}

\end{proof}

\section{Proof of the main result}\label{s1}

In this section, first we study a maximization problem. Then we prove our main result.

Fix $(\sigma,\textbf{Q}_{m})\in [0,2\pi]\times\Omega_{m},$ we define a new functional
\begin{equation}\label{4.1}
\mathcal {M}(\sigma,\textbf{Q}_{m}) = J(u_{\textbf{Q}_{m}}) = J(z_{\textbf{Q}_{m}} + \varphi_{\sigma,\textbf{Q}_{m}}) : [0,2\pi]\times\Omega_{m}\rightarrow \mathbb{R}.
\end{equation}
Since both $z_{\textbf{Q}_{m}}$ and $ \varphi_{\sigma,\textbf{Q}_{m}}$
are both $2\pi$-periodic respect to $\sigma,$ we only need to consider the
maximum problem of $\mathcal {M}(\sigma,\textbf{Q}_{m}) $ respect to $\textbf{Q}_{m}$ in $\Omega_{m}.$ So in the sequel, for simplicity
we denote $\mathcal {M}(\sigma,\textbf{Q}_{m})$ as $\mathcal {M}(\textbf{Q}_{m}).$

Define
\begin{equation}\label{4.2}
\mathcal {C}_{m}=\sup_{\textbf{Q}_{m}\in\Omega_{m}}\mathcal {M}(\textbf{Q}_{m})
\end{equation}

Note that $\mathcal {M}(\textbf{Q}_{m})$ is continuous in $\textbf{Q}_{m}$. We will show below that the maximization problem has a solution. Let $\mathcal
{M}(\bar{\textbf{Q}}_{m})$ be the maximum where $\bar{\textbf{Q}}_{m} = (\bar{Q}_{1}, \ldots, \bar{Q}_{m}) \in \bar{\Omega}_{m}$ that is
\begin{equation}\label{4.3}
\mathcal {M}(\bar{\textbf{Q}}_{m})=\max_{\textbf{Q}_{m}\in\Omega_{m}}\mathcal {M}(\textbf{Q}_{m})
\end{equation}
and we denote the solution by $u_{\bar{\textbf{Q}}_{m}}$.

First we prove that the maximum can be attained at finite points for each $\mathcal {C}_{m}$.
\begin{lem}\label{lem4.1}
Let assumptions $(A1)-(A4)$, $(V1)-(V2)$ and the assumptions in Proposition 2.4 be satisfied. Then, for all m:\\
(i)There exists $\textbf{Q}_{m} \in\Omega_{m}$ such that
\begin{equation}\label{4.4}
\mathcal {C}_{m} =\mathcal {M}(\textbf{Q}_{m});
\end{equation}
(ii) There holds
\begin{equation}\label{4.5}
\mathcal {C}_{m+1} > \mathcal {C}_{m} + I(z),
\end{equation}
where I(z) is the energy of the solution z of \eqref{ea0}:
\begin{equation}\label{4.6}
I(z)=\frac{1}{2}\int\Bigl|\Big(\frac{\nabla}{i}-A_{0}\Big)z\Bigl|^{2}+|z|^{2}-\int F(z)
\end{equation}
\end{lem}

\begin{proof}
We divide the proof into the following two steps.

$\textbf{Step 1}$: $\mathcal {C}_{1} > I(z)$, and $\mathcal {C}_{1}$ can be attained at a finite point. First applying standard Liapnunov-Schmidt reduction,
we have
\begin{equation}\label{4.7}
\|\varphi_{\sigma,Q}\|_{H^{1}}\leq C\|\epsilon \tilde{V}z_{Q}\|_{L^{2}}+C\bigl\|\epsilon^{2} |\tilde{A}|^{2}z_{Q}\bigl\|_{L^{2}}.
\end{equation}
Assuming that $|Q| \rightarrow\infty$, then we have

\begin{equation}\label{4.8}
\begin{array}{ll}
J(u_{Q})&=\ds\frac{1}{2}\int\Bigl|\Bigl(\frac{\nabla}{i}-A_{\epsilon}(x)\Bigl)u_{Q}\Bigl|^{2}+V_{\epsilon}(x)|u_{Q}|^{2}-\int F(u_{Q}) \vspace{0.2cm}\\
&=\ds\frac{1}{2}\int\Bigl|\Bigl(\frac{\nabla}{i}-A_{0}\Bigl) z_{Q}\Bigl|^{2}+\frac{1}{2}\int\Bigl|\Bigl(\frac{\nabla}{i}-A_{0}\Bigl) \varphi_{\sigma,Q}\Bigl|^{2}
+Re\int\Bigl(\frac{\nabla}{i}-A_{0}\Bigl) z_{Q}\overline{\Bigl(\frac{\nabla}{i}-A_{0}\Bigl)\varphi_{\sigma,Q}}\vspace{0.2cm}\\
&\,\,\,\,\,\,\ds+\frac{1}{2}\int\epsilon^{2}|\tilde{A}(x)|^{2}|z_{Q}|^{2}+\frac{1}{2}\int\epsilon^{2}|\tilde{A}(x)|^{2}|\varphi_{\sigma,Q}|^{2}+Re\int\epsilon^{2}|\tilde{A}(x)|^{2}z_{Q}\bar{\varphi}_{\sigma,Q}\vspace{0.2cm}\\
&\,\,\,\,\,\,\ds-Re\int \varepsilon \tilde{A}(x)\Bigl(\frac{\nabla w_{Q}}{i}\xi_{Q}\bar{\varphi}_{\sigma,Q}+\frac{\nabla \varphi_{\sigma,Q}}{i}\bar{z}_{Q}+\frac{\nabla \varphi_{\sigma,Q}}{i}\bar{\varphi}_{\sigma,Q} -A_{0}\varphi_{\sigma,Q}\bar{z}_{Q}-A_{0}\varphi_{\sigma,Q}\bar{\varphi}_{\sigma,Q}\Bigl)\vspace{0.2cm}\\
&\,\,\,\,\,\,\ds+\frac{1}{2}\int|z_{Q}|^{2}+\frac{1}{2}\int|\varphi_{\sigma,Q}|^{2}+Re\int z_{Q}\bar{\varphi}_{\sigma,Q}\vspace{0.2cm}\\
&\,\,\,\,\,\,\ds+\frac{1}{2}\int\epsilon\tilde{V}(x)\bigl(|z_{Q}|^{2}+|\varphi_{\sigma,Q}|^{2}+2Rez_{Q}\bar{\varphi}_{\sigma,Q}\bigl)-\int F(u_{Q})\vspace{0.2cm}\\
&=\ds\frac{1}{2}\int\Bigl|\Bigl(\frac{\nabla}{i}-A_{0}\Bigl) z_{Q}\Bigl|^{2}+\frac{1}{2}\int|z_{Q}|^{2}-\int F(z_{Q})+\frac{1}{2}\int\Bigl|\Bigl(\frac{\nabla}{i}-A_{0}\Bigl) \varphi_{\sigma,Q}\Bigl|^{2}+\frac{1}{2}\int|\varphi_{\sigma,Q}|^{2}\vspace{0.2cm}\\
&\,\,\,\,\,\,\ds+\int F(z_{Q})-\int F(u_{Q})+Re\int\Bigl(\frac{\nabla}{i}-A_{0}\Bigl) z_{Q}\overline{\Bigl(\frac{\nabla}{i}-A_{0}\Bigl)\varphi_{\sigma,Q}}+Re\int z_{Q}\bar{\varphi}_{\sigma,Q}\vspace{0.2cm}\\
&\,\,\,\,\,\,\ds+\frac{1}{2}\int\epsilon\tilde{V}(x)|z_{Q}|^{2}+\frac{1}{2}\int\epsilon^{2}|\tilde{A}(x)|^{2}|z_{Q}|^{2}+\frac{1}{2}\int\epsilon^{2}|\tilde{A}(x)|^{2}|\varphi_{\sigma,Q}|^{2}
\vspace{0.2cm}\\
&\,\,\,\,\,\,\ds+Re\int\epsilon^{2}|\tilde{A}(x)|^{2}z_{Q}\bar{\varphi}_{\sigma,Q}+\frac{1}{2}\int\epsilon\tilde{V}(x)|\varphi_{\sigma,Q}|^{2}+Re\int\epsilon\tilde{V}(x)z_{Q}\bar{\varphi}_{\sigma,Q}\vspace{0.2cm}\\
&\,\,\,\,\,\,\ds-Re\int \varepsilon \tilde{A}(x)\Bigl(\frac{\nabla w_{Q}}{i}\xi_{Q}\bar{\varphi}_{\sigma,Q}+\frac{\nabla \varphi_{\sigma,Q}}{i}\bar{z}_{Q}+\frac{\nabla \varphi_{\sigma,Q}}{i}\bar{\varphi}_{\sigma,Q} -A_{0}\varphi_{\sigma,Q}\bar{z}_{Q}-A_{0}\varphi_{\sigma,Q}\bar{\varphi}_{\sigma,Q}\Bigl)\vspace{0.2cm}\\
&\geq\ds I(z)+\frac{\epsilon}{4}\int \tilde{V}(x)|w_{Q}|^{2}+\frac{\epsilon^{2}}{4}\int |\tilde{A}(x)|^{2}|w_{Q}|^{2}-C\|\varphi_{\sigma,Q}\|_{H^{1}}^{2}-\delta\epsilon\int|div\tilde{A}(x)|^{2}|w_{Q}|^{2}\vspace{0.2cm}\\
&\geq\ds I(z)+\frac{\epsilon}{4}\int \tilde{V}(x)|w_{Q}|^{2}+\frac{\epsilon^{2}}{4}\int |\tilde{A}(x)|^{2}|w_{Q}|^{2}\vspace{0.2cm}\\
&\,\,\,\,\,\,\ds-\delta\epsilon\int|div\tilde{A}(x)|^{2}|w_{Q}|^{2}-\int\epsilon^{2} \tilde{V}^{2}(x)|w_{Q}|^{2}-\int\epsilon^{4} |\tilde{A}(x)|^{2}|w_{Q}|^{2}\vspace{0.2cm}\\
&\geq\ds I(z)+\frac{1}{8}\Bigl(\int_{B_{\frac{\rho}{2}}(Q)} \epsilon \tilde{V}(x)|w_{Q}|^{2}-\sup_{B_{\frac{|Q|}{4}}(0)}|w_{Q}|^{\frac{3}{2}}\int_{supp \tilde{V}^{-}} \epsilon|\tilde{V}(x)|w_{Q}^{\frac{1}{2}}\Bigl)\vspace{0.2cm}\\
&\,\,\,\,\,\,\ds+\frac{1}{8}\Bigl(\int_{B_{\frac{\rho}{2}}(Q)} \epsilon^{2} |\tilde{A}|^{2}|w_{Q}|^{2}-\sup_{B_{\frac{|Q|}{4}}(0)}|w_{Q}|^{\frac{3}{2}}\int_{supp \tilde{A}^{-}} \epsilon^{2}|\tilde{A}(x)|^{2}w_{Q}^{\frac{1}{2}}\vspace{0.2cm}\\
&\,\,\,\,\,\,\ds-\sup_{B_{\frac{|Q|}{4}}(0)}|w_{Q}|^{\frac{3}{2}}\int_{supp \tilde{A}^{-}} \epsilon^{2}|div\tilde{A}(x)|^{2}w_{Q}^{\frac{1}{2}}\Bigl)\vspace{0.2cm}\\
\end{array}
\end{equation}
\begin{equation*}
\begin{array}{ll}
&\geq\ds I(z)+\frac{1}{8}\int_{B_{\frac{\rho}{2}}(Q)} \epsilon \tilde{V}(x)|w_{Q}|^{2}+\frac{1}{8}\int_{B_{\frac{\rho}{2}}(Q)} \epsilon^{2} |\tilde{A}|^{2}|w_{Q}|^{2}-O(e^{-\frac{9}{8}|Q|}),
\end{array}
\end{equation*}
where we use the fact that
\begin{equation}\label{4.9}
\begin{array}{ll}
&\ds \frac{1}{2}\int\epsilon\tilde{V}(x)|z_{Q}|^{2}+\frac{1}{2}\int\epsilon^{2}|\tilde{A}(x)|^{2}|z_{Q}|^{2}
+Re\int\epsilon^{2}|\tilde{A}(x)|^{2}z_{Q}\bar{\varphi}_{\sigma,Q}+Re\int\epsilon\tilde{V}(x)z_{Q}\bar{\varphi}_{\sigma,Q}\vspace{0.2cm}\\
&\ds -Re\int \varepsilon \tilde{A}(x)\Bigl(\frac{\nabla w_{Q}}{i}\xi_{Q}\bar{\varphi}_{\sigma,Q}+\frac{\nabla \varphi_{\sigma,Q}}{i}\bar{z}_{Q}+\frac{\nabla \varphi_{\sigma,Q}}{i}\bar{\varphi}_{\sigma,Q} -A_{0}\varphi_{\sigma,Q}\bar{z}_{Q}-A_{0}\varphi_{\sigma,Q}\bar{\varphi}_{\sigma,Q}\Bigl)\vspace{0.2cm}\\
&\ds\leq \frac{\epsilon}{4}\int \tilde{V}(x)|w_{Q}|^{2}+\frac{\epsilon^{2}}{4}\int |\tilde{A}(x)|^{2}|w_{Q}|^{2}-C\|\varphi_{\sigma,Q}\|_{H^{1}}^{2}-\delta\epsilon\int|div\tilde{A}(x)|^{2}|w_{Q}|^{2}.
\end{array}
\end{equation}
By the slow decay assumption on the potential $\tilde{V}(x)$~and~$\tilde{A}(x)$, we have
$$\frac{1}{8}\int_{B_{\frac{\rho}{2}}(Q)} \epsilon \tilde{V}(x)|w_{Q}|^{2}+\frac{1}{8}\int_{B_{\frac{\rho}{2}}(Q)} \epsilon^{2} |\tilde{A}|^{2}|w_{Q}|^{2}-O(e^{-\frac{9}{8}|Q|})>0,~for~|Q|~large.$$
So
\begin{equation}\label{4.10}
\mathcal {C}_{1}\geq J(u_{Q})>I(z).
\end{equation}

Now we will prove that $\mathcal {C}_{1}$ can be attained at a finite point. Let ${Q_{j}}$ be a sequence such that $\lim\limits_{j\rightarrow\infty}\mathcal
{M}(Q_{j}) = \mathcal {C}_{1}$, and assume that $|Q_{j} |\rightarrow+\infty$,

\begin{equation}\label{4.11}
\begin{array}{ll}
J(u_{Q_{j}})&=\ds\frac{1}{2}\int\Bigl|(\frac{\nabla}{i}-A_{\epsilon}(x))u_{Q_{j}}\Bigl|^{2}+V_{\epsilon}(x)|u_{Q_{j}}|^{2}-\int F(u_{Q_{j}}) \vspace{0.2cm}\\
&=\ds I(z)+\frac{1}{2}\int\Bigl|\Bigl(\frac{\nabla}{i}-A_{0}\Bigl) \varphi_{\sigma,Q_{j}}\Bigl|^{2}+\frac{1}{2}\int|\varphi_{\sigma,Q_{j}}|^{2}\vspace{0.2cm}\\
&\,\,\,\,\,\,\ds+Re\int\Bigl(\frac{\nabla}{i}-A_{0}\Bigl) z_{Q_{j}}\overline{\Bigl(\frac{\nabla}{i}-A_{0}\Bigl)\varphi_{\sigma,Q_{j}}}+Re\int z_{Q_{j}}\bar{\varphi}_{\sigma,Q_{j}}-Re\int f(z_{Q_{j}})\bar{\varphi}_{\sigma,Q_{j}}\vspace{0.2cm}\\
&\,\,\,\,\,\,\ds-\int(F(u_{Q_{j}})-F(z_{Q_{j}}))+Re\int f(z_{Q_{j}})\bar{\varphi}_{\sigma,Q_{j}}+\frac{\epsilon}{2}\int \tilde{V}(x)|u_{Q_{j}}|^{2}+\frac{\epsilon}{2}\int |\tilde{A}(x)|^{2}|u_{Q_{j}}|^{2}\vspace{0.2cm}\\
&\,\,\,\,\,\,\ds-Re\int \varepsilon \tilde{A}(x)\Bigl(\frac{\nabla}{i}-A_{0}\Bigl)(z_{Q_{j}}+\varphi_{\sigma,Q_{j}})\overline{(z_{Q_{j}}+\varphi_{\sigma,Q_{j}})}\vspace{0.2cm}\\
&\leq\ds I(z)+C\|\varphi_{\sigma,Q_{j}}\|^{2}+\frac{\epsilon}{2}\int \tilde{V}(x)|u_{Q_{j}}|^{2}+\frac{\epsilon}{2}\int |\tilde{A}(x)|^{2}|u_{Q_{j}}|^{2}\vspace{0.2cm}\\
&\leq\ds I(z)+O\Big(\int\epsilon^{2}\tilde{V}^{2}|z_{Q_{j}}|^{2}\Big)+\frac{\epsilon}{2}\int \tilde{V}(x)|u_{Q_{j}}|^{2}+\frac{\epsilon}{2}\int |\tilde{A}(x)|^{2}|u_{Q_{j}}|^{2}.
\end{array}
\end{equation}

Since $|\tilde{A}(x)| \rightarrow 0$ and $\tilde{V}(x) \rightarrow 0$ as $|x|\rightarrow \infty$, we have
$$O\Big(\int\epsilon^{2}\tilde{V}^{2}|z_{Q_{j}}|^{2}\Big)+\frac{\epsilon}{2}\int \tilde{V}(x)|u_{Q_{j}}|^{2}+\frac{\epsilon}{2}\int |\tilde{A}(x)|^{2}|u_{Q_{j}}|^{2}\rightarrow0,$$

so we have
$$\mathcal {C}_{1}=\lim\limits_{j\rightarrow\infty}J(u_{Q_{j}})\leq I(z),$$
which contradicts to \eqref{4.10}. Thus $\mathcal {C}_{1}$ can be attained at a finite point.

$\textbf{Step 2}$: Assume that there exists $\bar{\textbf{Q}}_{m} = (\bar{Q}_{1},\ldots, \bar{Q}_{m}) \in\Omega_{m}$ such that $\mathcal {C}_{m} =\mathcal
{M}(\textbf{Q}_{m})$ and we denote the solution by $u_{\bar{\textbf{Q}}_{m}}$. Next we prove that there exists $\textbf{Q}_{m+1} \in\Omega_{m+1}$ such that
$\mathcal {C}_{m+1}$ can be attained. Let $\textbf{Q}_{m+1}^{(n)}$ be a sequence such that
\begin{equation}\label{4.12}
\mathcal {C}_{m+1}=\lim\limits_{n\rightarrow\infty}\mathcal {M}(\textbf{Q}_{m+1}^{(n)}).
\end{equation}

We claim that $\textbf{Q}_{m+1}^{(n)}$ is bounded. We prove it by contradiction. In the following we omit index $n$ for simplicity. By direct computation, we
have

\begin{equation}\label{4.13}
\begin{array}{ll}
&J(u_{\textbf{Q}_{m+1}})=J(u_{\textbf{Q}_{m}}+z_{Q_{m+1}}+\phi_{m+1}) \vspace{0.2cm}\\
&=\ds\frac{1}{2}\int\Bigl|\Big(\frac{\nabla}{i}-A_{\epsilon}(x)\Big)(u_{\textbf{Q}_{m}}+z_{Q_{m+1}}+\phi_{m+1})\Bigl|^{2}+V_{\epsilon}(x)|u_{\textbf{Q}_{m}}+z_{Q_{m+1}}+\phi_{m+1}|^{2}\vspace{0.2cm}\\
&\,\,\,\,\,\,\ds-\int F(u_{\textbf{Q}_{m}}+z_{Q_{m+1}}+\phi_{m+1}) \vspace{0.2cm}\\
&=\ds J(u_{\textbf{Q}_{m}}+z_{Q_{m+1}})+Re\int\Bigl(\frac{\nabla}{i}-A_{\epsilon}(x)\Bigl)u_{\textbf{Q}_{m}}\overline{\Bigl(\frac{\nabla}{i}-A_{\epsilon}(x)\Bigl)\phi_{m+1}}\vspace{0.2cm}\\
&\,\,\,\,\,\,\ds +Re\int V_{\epsilon}(x)u_{\textbf{Q}_{m}}\bar{\phi}_{m+1}-Re\int
f(u_{\textbf{Q}_{m}})\bar{\phi}_{m+1}\ds+Re\int f(u_{\textbf{Q}_{m}})\bar{\phi}_{m+1}-Re\int f(u_{\textbf{Q}_{m}}+z_{Q_{m+1}})\bar{\phi}_{m+1}\vspace{0.2cm}\\
&\,\,\,\,\,\,\ds-\int F(u_{\textbf{Q}_{m}}+z_{Q_{m+1}}+\phi_{m+1})+\int F(u_{\textbf{Q}_{m}}+z_{Q_{m+1}})+Re\int
f(u_{\textbf{Q}_{m}}+z_{Q_{m+1}})\bar{\phi}_{m+1}\vspace{0.2cm}\\
&\,\,\,\,\,\,\ds+Re\int\Bigl(\frac{\nabla}{i}-A_{\epsilon}(x)\Bigl)z_{Q_{m+1}}\overline{\Bigl(\frac{\nabla}{i}-A_{\epsilon}(x)\Bigl)\phi_{m+1}}+\frac{1}{2}\int V_{\epsilon}(x)|\phi_{m+1}|^{2}\vspace{0.2cm}\\
&\,\,\,\,\,\,\ds +\frac{1}{2}\int\Bigl|(\frac{\nabla}{i}-A_{\epsilon}(x))\phi_{m+1}\Bigl|^{2}+Re\int V_{\epsilon}(x)z_{Q_{m+1}}\bar{\phi}_{m+1}\vspace{0.2cm}\\
&=\ds J(u_{\textbf{Q}_{m}}+z_{Q_{m+1}})-\int\sum_{j=1}^{m}\sum_{k=1}^{N+1}c_{jk}D_{jk}\phi_{m+1}\vspace{0.2cm}\\
&\,\,\,\,\,\,\ds+Re\int f(u_{\textbf{Q}_{m}})\bar{\phi}_{m+1}-Re\int f(u_{\textbf{Q}_{m}}+z_{Q_{m+1}})\bar{\phi}_{m+1}+Re\int f(z_{Q_{m+1}})\bar{\phi}_{m+1}\vspace{0.2cm}\\
&\,\,\,\,\,\,\ds-\int f'(u_{\textbf{Q}_{m}}+z_{Q_{m+1}}+\vartheta\phi_{m+1})|\phi_{m+1}|^{2}+\frac{1}{2}\int\Bigl|\Big(\frac{\nabla}{i}-A_{\epsilon}(x)\Big)\phi_{m+1}\Bigl|^{2}\vspace{0.2cm}\\
&\,\,\,\,\,\,\ds+Re\int\Bigl(\frac{\nabla}{i}-A_{\epsilon}(x)\Bigl)z_{Q_{m+1}}\overline{\Bigl(\frac{\nabla}{i}-A_{\epsilon}(x)\Bigl)\phi_{m+1}}\vspace{0.2cm}\\
&\,\,\,\,\,\,\ds +\frac{1}{2}\int V_{\epsilon}(x)|\phi_{m+1}|^{2}+Re\int V_{\epsilon}(x)z_{Q_{m+1}}\bar{\phi}_{m+1}-Re\int f(z_{Q_{m+1}})\bar{\phi}_{m+1}\vspace{0.2cm}\\
&=\ds J(u_{\textbf{Q}_{m}}+z_{Q_{m+1}})+O(\|\phi_{m+1}\|^{2}+\|\mathcal
{\bar{S}}(u_{\textbf{Q}_{m}}+z_{Q_{m+1}})\|\|\phi_{m+1}\|)-\ds\int\sum_{j=1}^{m}\sum_{k=1}^{N+1}c_{j,k}D_{j,k}\phi_{m+1}\vspace{0.2cm}\\
&=\ds J(u_{\textbf{Q}_{m}}+z_{Q_{m+1}})+O\Bigl[e^{-\beta\mu}\sum_{j=1}^{m}w(|Q_{m+1}-Q_{j}|)+\epsilon^{2}\bigl(\int|\tilde{V}(x)||w_{Q_{m+1}}|\bigl)^{2}\vspace{0.2cm}\\
&\,\,\,\,\,\,\ds+\epsilon^{4}\bigl(\int|\tilde{A}(x)|^{2}|w_{Q_{m+1}}|\bigl)^{2}+\epsilon^{2}\bigl( \int|\tilde{A}(x)||\nabla
w_{Q_{m+1}}|\bigl)^{2}+\epsilon^{2}\bigl(\int|div\tilde{A}(x)||w_{Q_{m+1}}| \bigl)^{2}\vspace{0.2cm}\\
&\,\,\,\,\,\,\ds+\epsilon^{2}\int|\tilde{V}(x)|^{2}|w_{Q_{m+1}}|^{2}+\epsilon^{2}\int|\tilde{A}(x)|^{2}|\nabla
w_{Q_{m+1}}|^{2}+\epsilon^{2}\int|div\tilde{A}(x)|^{2}|w_{Q_{m+1}}|^{2}\vspace{0.2cm}\\
&\,\,\,\,\,\,\ds+\epsilon^{4}\int|\tilde{A}(x)|^{4}|w_{Q_{m+1}}|^{2}\Bigl].
\end{array}
\end{equation}

Moreover, we have
\begin{equation}\label{4.14}
\begin{array}{ll}
&J(u_{\textbf{Q}_{m}}+z_{Q_{m+1}})\vspace{0.2cm}\\
&=\ds\frac{1}{2}\int\Bigl|\Big(\frac{\nabla}{i}-A_{\epsilon}(x)\Big)(u_{\textbf{Q}_{m}}+z_{Q_{m+1}})\Bigl|^{2}
+V_{\epsilon}(x)|u_{\textbf{Q}_{m}}+z_{Q_{m+1}}|^{2}
-\int F(u_{\textbf{Q}_{m}}+z_{Q_{m+1}}) \vspace{0.2cm}\\
&\leq\ds \mathcal {C}_{m}+\frac{1}{2}\int|z_{Q_{m+1}}|^{2}+\frac{1}{2}\int\Bigl|(\frac{\nabla}{i}-A_{0})z_{Q_{m+1}}\Bigl|^{2}-\int F(z_{Q_{m+1}})\vspace{0.2cm}\\
&\,\,\,\,\,\,\ds +Re\int(1+\epsilon \tilde{V}(x))u_{\textbf{Q}_{m}}\bar{z}_{Q_{m+1}}
+Re\int\Bigl(\frac{\nabla}{i}-A_{\epsilon}(x)\Bigl)u_{\textbf{Q}_{m}}\overline{\Bigl(\frac{\nabla}{i}-A_{\epsilon}(x)\Bigl)z_{Q_{m+1}}}\vspace{0.2cm}\\
&\,\,\,\,\,\,\ds-\int F(u_{\textbf{Q}_{m}}+z_{Q_{m+1}})+\int F(u_{\textbf{Q}_{m}})+\int F(z_{Q_{m+1}})+\frac{1}{2}\int\epsilon \tilde{V}(x)|z_{Q_{m+1}}|^{2}\vspace{0.2cm}\\
&\,\,\,\,\,\,\ds+\frac{1}{2}\int\epsilon^{2} |\tilde{A}(x)|^{2}|z_{Q_{m+1}}|^{2}-Re\int\epsilon \tilde{A}(x)\Bigl(\frac{\nabla}{i}-A_{0}\Bigl)z_{Q_{m+1}} \bar{z}_{Q_{m+1}}\vspace{0.2cm}\\
&\leq\ds \mathcal {C}_{m}+I(z)+\frac{1}{2}\int\epsilon \tilde{V}(x)|z_{Q_{m+1}}|^{2}+\frac{1}{2}\int\epsilon^{2} |\tilde{A}(x)|^{2}|z_{Q_{m+1}}|^{2}\vspace{0.2cm}\\
&\,\,\,\,\,\,\ds+Re\int \bigl(f(u_{\textbf{Q}_{m}})-\sum_{j=1}^{m}\sum_{k=1}^{N+1}c_{jk}D_{j,k}\bigl)\bar{z}_{Q_{m+1}}-Re\int f(u_{\textbf{Q}_{m}})\bar{z}_{Q_{m+1}}\vspace{0.2cm}\\
&\,\,\,\,\,\,\ds-Re\int f(z_{Q_{m+1}})\bar{u}_{\textbf{Q}_{m}}+O\Big(e^{-\beta\mu}\sum_{j=1}^{m}w(|Q_{m+1}-Q_{j}|)\Big)\vspace{0.2cm}\\
&\leq\ds \mathcal {C}_{m}+I(z)+\frac{1}{2}\int\epsilon \tilde{V}(x)|z_{Q_{m+1}}|^{2}+\frac{1}{2}\int\epsilon^{2} |\tilde{A}(x)|^{2}|z_{Q_{m+1}}|^{2}
-Re\int f(z_{Q_{m+1}})\bar{u}_{\textbf{Q}_{m}}\vspace{0.2cm}\\
&\quad-Re\ds\int \sum_{j=1}^{m}\sum_{k=1}^{N+1}c_{j,k}D_{j,k}\bar{z}_{Q_{m+1}}
+O\Big(e^{-\beta\mu}\ds\sum_{j=1}^{m}w(|Q_{m+1}-Q_{j}|)\Big).
\end{array}
\end{equation}
By estimate \eqref{2.28} in Proposition \ref{prop2.3}, and that the definition of $D_{j,k}$, we have
\begin{equation}\label{4.15}
\Bigl|Re\int\sum_{j=1}^{m}\sum_{k=1}^{N+1}c_{j,k}D_{j,k}\bar{z}_{Q_{m+1}}\Bigl|\leq Ce^{-\beta\mu}\sum_{j=1}^{m}w(|Q_{m+1}-Q_{j}|).
\end{equation}
By the equation satisfied by $\varphi_{\sigma,\textbf{Q}_{m}}$
\begin{equation}\label{4.16}
L(\varphi_{\sigma,\textbf{Q}_{m}})=-\mathcal {S}(z_{\textbf{Q}_{m}})+\mathcal {N}(\varphi_{\sigma,\textbf{Q}_{m}})+\sum_{j=1}^{m}\sum_{k=1}^{N+1}c_{j,k}D_{j,k},
\end{equation}
we have
\begin{equation}\label{4.17}
\begin{array}{ll}
&\ds Re\int f(z_{Q_{m+1}})\bar{\varphi}_{\sigma,\textbf{Q}_{m}}\vspace{0.2cm}\\
&=\ds Re\int \Bigl(\frac{\nabla}{i}-A_{0}\Bigl)z_{Q_{m+1}}\overline{\Bigl(\frac{\nabla}{i}-A_{0}\Bigl)\varphi_{\sigma,\textbf{Q}_{m}}}+Re\int z_{\textbf{Q}_{m}}\bar{\varphi}_{Q_{m+1}}\vspace{0.2cm}\\
&=\ds Re\int \Bigl(\frac{\nabla}{i}-A_{0}\Bigl)\varphi_{\sigma,\textbf{Q}_{m}}\overline{\Bigl(\frac{\nabla}{i}-A_{0}\Bigl)z_{Q_{m+1}}}+Re\int \varphi_{\sigma,\textbf{Q}_{m}}\bar{z}_{Q_{m+1}}\vspace{0.2cm}\\
&=\ds Re\int\Bigl(\mathcal {S}(z_{\textbf{Q}_{m}})-\mathcal {N}(\varphi_{\sigma,\textbf{Q}_{m}})-\sum_{j=1}^{m}\sum_{k=1}^{N+1}c_{j,k}D_{j,k}-\epsilon \tilde{V}\varphi_{\sigma,\textbf{Q}_{m}} \vspace{0.2cm}\\
\end{array}
\end{equation}
\begin{equation*}
\begin{array}{ll}
&\,\,\,\,\,\,\ds+f'(z_{\textbf{Q}_{m}})\varphi_{\sigma,\textbf{Q}_{m}}\Bigl)\bar{z}_{Q_{m+1}}+Re\int \Bigl(\frac{\nabla}{i}-A_{0}\Bigl)\varphi_{\sigma,\textbf{Q}_{m}}\overline{\Bigl(\frac{\nabla}{i}-A_{0}\Bigl)z_{Q_{m+1}}}\vspace{0.2cm}\\
&\,\,\,\,\,\,\ds-Re\int\bigl(\frac{\nabla}{i}-A_{\epsilon}(x)\bigl)\varphi_{\sigma,\textbf{Q}_{m}}\overline{\bigl(\frac{\nabla}{i}-A_{\epsilon}(x)\bigl)z_{Q_{m+1}}}\vspace{0.2cm}\\
&=\ds Re\int\Bigl(\mathcal {S}(z_{\textbf{Q}_{m}})-\mathcal {N}(\varphi_{\sigma,\textbf{Q}_{m}})-\sum_{j=1}^{m}\sum_{k=1}^{N+1}c_{j,k}D_{j,k}-\epsilon \tilde{V}\varphi_{\sigma,\textbf{Q}_{m}} \vspace{0.2cm}\\
&\,\,\,\,\,\,\ds+f'(z_{\textbf{Q}_{m}})\varphi_{\sigma,\textbf{Q}_{m}}\Bigl)\bar{z}_{Q_{m+1}}+Re\int\epsilon
\tilde{A}\Bigl(\frac{\nabla}{i}-A_{0}\Bigl)\varphi_{\sigma,\textbf{Q}_{m}}\bar{z}_{Q_{m+1}}\vspace{0.2cm}\\
&\,\,\,\,\,\,\ds+Re\int\epsilon \tilde{A}\varphi_{\sigma,\textbf{Q}_{m}}\overline{\Bigl(\frac{\nabla}{i}-A_{0}\Bigl)z_{Q_{m+1}}}-Re\int\epsilon^{2}
|\tilde{A}|^{2}\varphi_{\sigma,\textbf{Q}_{m}}\bar{z}_{Q_{m+1}}.
\end{array}
\end{equation*}

Moreover, we can choose $\gamma$ that $\gamma+\delta>1$, $(1+\delta)\gamma>1$. Then we can easily get
\begin{equation}\label{4.18}
\Bigl|Re\int(\mathcal {N}(\varphi_{\sigma,\textbf{Q}_{m}})-f'(z_{\textbf{Q}_{m}})\varphi_{\sigma,\textbf{Q}_{m}})\bar{z}_{Q_{m+1}}\Bigl|\leq
Ce^{-\beta\mu}\sum_{j=1}^{m}w(|Q_{m+1}-Q_{j}|)
\end{equation}
and
\begin{equation}\label{4.19}
\begin{array}{ll}
&\ds \Bigl|Re\int(\mathcal {S}(z_{\textbf{Q}_{m}})-\epsilon \tilde{V}\varphi_{\sigma,\textbf{Q}_{m}})\bar{z}_{Q_{m+1}}+Re\int\epsilon
\tilde{A}\Bigl(\frac{\nabla}{i}-A_{0}\Bigl)\varphi_{\sigma,\textbf{Q}_{m}}\bar{z}_{Q_{m+1}}\vspace{0.2cm}\\
&\,\,\,\,\,\,\ds +Re\int\epsilon \tilde{A}\varphi_{\sigma,\textbf{Q}_{m}}\overline{\Bigl(\frac{\nabla}{i}-A_{0}\Bigl)z_{Q_{m+1}}}-Re\int\epsilon^{2}
|\tilde{A}|^{2}\varphi_{\sigma,\textbf{Q}_{m}}\bar{z}_{Q_{m+1}}\Bigl|\vspace{0.2cm}\\
&=\ds \Bigl|Re\int\Bigl(f(z_{\textbf{Q}_{m}})-\sum_{j=1}^{m}f(z_{Q_{j}})-\epsilon \tilde{V}\varphi_{\sigma,\textbf{Q}_{m}}-\epsilon \tilde{V}z_{\textbf{Q}_{m}}\vspace{0.2cm}\\
&\,\,\,\,\,\,\ds+\frac{\epsilon}{i}div\tilde{A}(x)z_{\textbf{Q}_{m}}+\frac{2\epsilon}{i}\sum_{j=1}^{m}\xi_{j}\tilde{A}(x)\cdot\nabla w_{Q_{j}}-\epsilon^{2}|\tilde{A}(x)|^{2}z_{\textbf{Q}_{m}}\Bigl)\bar{z}_{Q_{m+1}}\vspace{0.2cm}\\
&\,\,\,\,\,\,\ds+Re\int\frac{\epsilon}{i}\tilde{A}(x)\cdot\nabla\varphi_{\sigma,\textbf{Q}_{m}}\bar{z}_{Q_{m+1}}-Re\int\epsilon A_{0}\cdot \tilde{A}(x)
\varphi_{\sigma,\textbf{Q}_{m}}\bar{z}_{Q_{m+1}}\vspace{0.2cm}\\
&\,\,\,\,\,\,\ds-Re\int\frac{\epsilon}{i}\varphi_{\sigma,\textbf{Q}_{m}}\overline{\xi}_{\textbf{Q}_{m}}\tilde{A}(x)\cdot\nabla w_{Q_{m+1}}-Re\int\epsilon^{2}|\tilde{A}(x)|^{2}\varphi_{\sigma,\textbf{Q}_{m}}\bar{z}_{Q_{m+1}}\Bigl|\vspace{0.2cm}\\
&\leq\ds C\Bigl(\epsilon \int \tilde{V}w_{\textbf{Q}_{m}}w_{Q_{m+1}}+ \epsilon e^{-\beta\mu} \int
\sum_{j=1}^{m}e^{-\gamma|x-Q_{j}|}\tilde{V}w_{Q_{m+1}}+e^{-\beta\mu}\sum_{j=1}^{m}w(|Q_{m+1}-Q_{j}|)\vspace{0.2cm}\\
&\,\,\,\,\,\,\ds+ \epsilon \int |div\tilde{A}(x)|w_{\textbf{Q}_{m}}w_{Q_{m+1}} + \epsilon \int |\tilde{A}(x)||\nabla w_{\textbf{Q}_{m}}|w_{Q_{m+1}}+ \epsilon^{2} \int
|\tilde{A}(x)|^{2}w_{\textbf{Q}_{m}}w_{Q_{m+1}}\vspace{0.2cm}\\
&\,\,\,\,\,\,\ds+ \epsilon e^{-\beta\mu}\int |\tilde{A}(x)|\sum_{j=1}^{m}e^{-\gamma|x-Q_{j}|}w_{Q_{m+1}} + \epsilon e^{-\beta\mu}\int
|\tilde{A}(x)|\sum_{j=1}^{m}e^{-\gamma|x-Q_{j}|}w_{Q_{m+1}}\vspace{0.2cm}\\
&\,\,\,\,\,\,\ds+ \epsilon e^{-\beta\mu}\int |\tilde{A}(x)|\sum_{j=1}^{m}e^{-\gamma|x-Q_{j}|}|\nabla w_{Q_{m+1}}|+ \epsilon^{2} e^{-\beta\mu}\int
|\tilde{A}(x)|^{2}\sum_{j=1}^{m}e^{-\gamma|x-Q_{j}|}w_{Q_{m+1}}\Bigl).
\end{array}
\end{equation}

From \eqref{4.16} to \eqref{4.19}, we obtain
\begin{equation}\label{4.20}
\begin{array}{ll}
&\ds \Bigl|Re\int f(z_{Q_{m+1}})\bar{\varphi}_{\sigma,\textbf{Q}_{m}}\Bigl|\vspace{0.2cm}\\
&\leq\ds C\Bigl(\epsilon \int \tilde{V}w_{\textbf{Q}_{m}}w_{Q_{m+1}}+ \epsilon e^{-\beta\mu} \int
\sum_{j=1}^{m}e^{-\gamma|x-Q_{j}|}\tilde{V}w_{Q_{m+1}}+e^{-\beta\mu}\sum_{j=1}^{m}w(|Q_{m+1}-Q_{j}|)\vspace{0.2cm}\\
&\,\,\,\,\,\,\ds+ \epsilon \int |div\tilde{A}(x)|w_{\textbf{Q}_{m}}w_{Q_{m+1}} + \epsilon \int |\tilde{A}(x)||\nabla w_{\textbf{Q}_{m}}|w_{Q_{m+1}}+ \epsilon^{2} \int
|\tilde{A}(x)|^{2}w_{\textbf{Q}_{m}}w_{Q_{m+1}}\vspace{0.2cm}\\
&\,\,\,\,\,\,\ds+ \epsilon e^{-\beta\mu}\int |\tilde{A}(x)|\sum_{j=1}^{m}e^{-\gamma|x-Q_{j}|}w_{Q_{m+1}} + \epsilon e^{-\beta\mu}\int
|\tilde{A}(x)|\sum_{j=1}^{m}e^{-\gamma|x-Q_{j}|}w_{Q_{m+1}}\vspace{0.2cm}\\
&\,\,\,\,\,\,\ds+ \epsilon e^{-\beta\mu}\int |\tilde{A}(x)|\sum_{j=1}^{m}e^{-\gamma|x-Q_{j}|}|\nabla w_{Q_{m+1}}|+ \epsilon^{2} e^{-\beta\mu}\int
|\tilde{A}(x)|^{2}\sum_{j=1}^{m}e^{-\gamma|x-Q_{j}|}w_{Q_{m+1}}\Bigl).
\end{array}
\end{equation}
Hence by Lemma \ref{lem3.1}, we have
\begin{equation}\label{4.21}
\begin{array}{ll}
&\ds Re\int f(z_{Q_{m+1}})\bar{u}_{\textbf{Q}_{m}}=Re\int f(z_{Q_{m+1}})(\bar{z}_{\textbf{Q}_{m}}+\bar{\varphi}_{\sigma,\textbf{Q}_{m}})\vspace{0.2cm}\\
&\geq\ds\frac{1}{4}\vartheta\sum_{j=1}^{m}w(|Q_{m+1}-Q_{j}|)\vspace{0.2cm}\\
&\quad\ds+ O\Bigl(\epsilon \int \tilde{V}w_{\textbf{Q}_{m}}w_{Q_{m+1}}+ \epsilon e^{-\beta\mu} \int
\sum_{j=1}^{m}e^{-\gamma|x-Q_{j}|}\tilde{V}w_{Q_{m+1}}+e^{-\beta\mu}\sum_{j=1}^{m}w(|Q_{m+1}-Q_{j}|) \vspace{0.2cm}\\
&\quad\ds+ \epsilon \int |div\tilde{A}(x)|w_{\textbf{Q}_{m}}w_{Q_{m+1}} + \epsilon \int |\tilde{A}(x)||\nabla w_{\textbf{Q}_{m}}|w_{Q_{m+1}}+ \epsilon^{2} \int
|\tilde{A}(x)|^{2}w_{\textbf{Q}_{m}}w_{Q_{m+1}}\vspace{0.2cm}\\
&\quad\ds+ \epsilon e^{-\beta\mu}\int |\tilde{A}(x)|\sum_{j=1}^{m}e^{-\gamma|x-Q_{j}|}w_{Q_{m+1}} + \epsilon e^{-\beta\mu}\int
|\tilde{A}(x)|\sum_{j=1}^{m}e^{-\gamma|x-Q_{j}|}w_{Q_{m+1}}\vspace{0.2cm}\\
&\quad\ds+ \epsilon e^{-\beta\mu}\int |\tilde{A}(x)|\sum_{j=1}^{m}e^{-\gamma|x-Q_{j}|}|\nabla w_{Q_{m+1}}|+ \epsilon^{2} e^{-\beta\mu}\int
|\tilde{A}(x)|^{2}\sum_{j=1}^{m}e^{-\gamma|x-Q_{j}|}w_{Q_{m+1}}\Bigl).
\end{array}
\end{equation}

Combing \eqref{4.13}, \eqref{4.14}, \eqref{4.15} and \eqref{4.21}, we obtain
\begin{equation}\label{4.22}
\begin{array}{ll}
&\ds J(u_{\textbf{Q}_{m+1}})=J(u_{\textbf{Q}_{m}}+z_{Q_{m+1}}+\phi_{m+1}) \vspace{0.2cm}\\
&\leq\ds \mathcal {C}_{m}+I(z)+\frac{1}{2}\int\epsilon \tilde{V}(x)|z_{Q_{m+1}}|^{2}+\frac{1}{2}\int\epsilon^{2} |\tilde{A}(x)|^{2}|z_{Q_{m+1}}|^{2}-\frac{1}{4}\vartheta\sum_{j=1}^{m}w(|Q_{m+1}-Q_{j}|)\vspace{0.2cm}\\
&+\ds O\Bigl[\epsilon \int \tilde{V}w_{\textbf{Q}_{m}}w_{Q_{m+1}}+ \epsilon e^{-\beta\mu} \int
\sum_{j=1}^{m}e^{-\gamma|x-Q_{j}|}\tilde{V}w_{Q_{m+1}}+e^{-\beta\mu}\sum_{j=1}^{m}w(|Q_{m+1}-Q_{j}|)\vspace{0.2cm}\\
&\quad\quad\ds+ \epsilon \int |div\tilde{A}(x)|w_{\textbf{Q}_{m}}w_{Q_{m+1}} + \epsilon \int |\tilde{A}(x)||\nabla w_{\textbf{Q}_{m}}|w_{Q_{m+1}}+ \epsilon^{2} \int
|\tilde{A}(x)|^{2}w_{\textbf{Q}_{m}}w_{Q_{m+1}}\vspace{0.2cm}\\
&\quad\quad\ds+ \epsilon e^{-\beta\mu}\int |\tilde{A}(x)|\sum_{j=1}^{m}e^{-\gamma|x-Q_{j}|}w_{Q_{m+1}} + \epsilon e^{-\beta\mu}\int
|\tilde{A}(x)|\sum_{j=1}^{m}e^{-\gamma|x-Q_{j}|}w_{Q_{m+1}}\vspace{0.2cm}\\
\end{array}
\end{equation}
\begin{equation*}
\begin{array}{ll}
&\,\,\,\,\,\,\ds+ \epsilon e^{-\beta\mu}\int |\tilde{A}(x)|\sum_{j=1}^{m}e^{-\gamma|x-Q_{j}|}|\nabla w_{Q_{m+1}}|+ \epsilon^{2} e^{-\beta\mu}\int
|\tilde{A}(x)|^{2}\sum_{j=1}^{m}e^{-\gamma|x-Q_{j}|}w_{Q_{m+1}}\vspace{0.2cm}\\
&\,\,\,\,\,\,\ds+\epsilon^{4}\Bigl(\int|\tilde{A}(x)|^{2}|w_{Q_{m+1}}|\Bigl)^{2}+\epsilon^{2}\Bigl( \int|\tilde{A}(x)||\nabla
w_{Q_{m+1}}|\Bigl)^{2}+\epsilon^{2}\Bigl(\int|div\tilde{A}(x)||w_{Q_{m+1}}| \Bigl)^{2}\vspace{0.2cm}\\
&\,\,\,\,\,\,\ds+\epsilon^{2}\int|\tilde{V}(x)|^{2}|w_{Q_{m+1}}|^{2}+\epsilon^{2}\int|\tilde{A}(x)|^{2}|\nabla
w_{Q_{m+1}}|^{2}+\epsilon^{2}\int|div\tilde{A}(x)|^{2}|w_{Q_{m+1}}|^{2}\vspace{0.2cm}\\
&\,\,\,\,\,\,\ds+\epsilon^{4}\int|\tilde{A}(x)|^{4}|w_{Q_{m+1}}|^{2}+\epsilon^{2}\bigl(\int|\tilde{V}(x)||w_{Q_{m+1}}|\bigl)^{2}\Bigl].
\end{array}
\end{equation*}
By the assumption that $|Q^{(n)}_{ m+1}|\rightarrow+\infty$,
\begin{equation}\label{4.23}
\begin{array}{ll}
&\ds\epsilon \int \tilde{V}w_{\textbf{Q}_{m}}w_{Q^{(n)}_{m+1}}+ \epsilon e^{-\beta\mu} \int
\sum_{j=1}^{m}e^{-\gamma|x-Q_{j}|}\tilde{V}w_{Q^{(n)}_{m+1}}\vspace{0.2cm}\\
&\ds+ \epsilon \int |div\tilde{A}(x)|w_{\textbf{Q}_{m}}w_{Q^{(n)}_{m+1}} + \epsilon \int |\tilde{A}(x)||\nabla w_{\textbf{Q}_{m}}|w_{Q^{(n)}_{m+1}}+ \epsilon^{2}
\int
|\tilde{A}(x)|^{2}w_{\textbf{Q}_{m}}w_{Q^{(n)}_{m+1}}\vspace{0.2cm}\\
&\ds+ \epsilon e^{-\beta\mu}\int |\tilde{A}(x)|\sum_{j=1}^{m}e^{-\gamma|x-Q_{j}|}w_{Q^{(n)}_{m+1}} + \epsilon e^{-\beta\mu}\int
|\tilde{A}(x)|\sum_{j=1}^{m}e^{-\gamma|x-Q_{j}|}w_{Q^{(n)}_{m+1}}\vspace{0.2cm}\\
&\ds+ \epsilon e^{-\beta\mu}\int |\tilde{A}(x)|\sum_{j=1}^{m}e^{-\gamma|x-Q_{j}|}|\nabla w_{Q^{(n)}_{m+1}}|+ \epsilon^{2} e^{-\beta\mu}\int
|\tilde{A}(x)|^{2}\sum_{j=1}^{m}e^{-\gamma|x-Q_{j}|}w_{Q^{(n)}_{m+1}}\vspace{0.2cm}\\
&\ds+\epsilon^{4}\bigl(\int|\tilde{A}(x)|^{2}|w_{Q^{(n)}_{m+1}}|\bigl)^{2}+\epsilon^{2}\bigl( \int|\tilde{A}(x)||\nabla
w_{Q^{(n)}_{m+1}}|\bigl)^{2}+\epsilon^{2}\bigl(\int|div\tilde{A}(x)||w_{Q^{(n)}_{m+1}}| \bigl)^{2}\vspace{0.2cm}\\
&\ds+\epsilon^{2}\int|\tilde{V}(x)|^{2}|w_{Q^{(n)}_{m+1}}|^{2}+\epsilon^{2}\int|\tilde{A}(x)|^{2}|\nabla
w_{Q^{(n)}_{m+1}}|^{2}+\epsilon^{2}\int|div\tilde{A}(x)|^{2}|w_{Q^{(n)}_{m+1}}|^{2}\vspace{0.2cm}\\
&\ds+\epsilon^{4}\int|\tilde{A}(x)|^{4}|w_{Q^{(n)}_{m+1}}|^{2}+\epsilon^{2}
\Bigl(\int|\tilde{V}(x)||w_{Q^{(n)}_{m+1}}|\Bigl)^{2}\rightarrow0,~as~n\rightarrow+\infty
\end{array}
\end{equation}

and
\begin{equation}\label{4.24}
-\frac{1}{4}\vartheta\sum_{j=1}^{m}w(|Q_{m+1}-Q_{j}|)+O\Bigl(e^{-\beta\mu}\sum_{j=1}^{m}w(|Q_{m+1}-Q_{j}|)\Bigl)<0.
\end{equation}

Combining \eqref{4.12}, \eqref{4.22}, \eqref{4.23} and \eqref{4.24}, we have
\begin{equation}\label{4.25}
\mathcal {C}_{m+1}\leq\mathcal {C}_{m} + I(z).
\end{equation}

On the other hand, since by the assumption, $\mathcal {C}_{m}$ can be attained at $(\bar{Q}_{1},\ldots, \bar{Q}_{m})$, so there exists other point
$Q_{m+1}$ which is far away from the m points which be determined later. Next let's consider the solution concentrated at the points $(\bar{Q}_{1},\ldots,
\bar{Q}_{m},Q_{m+1})$, and we denote the solution by $u_{\bar{\textbf{Q}}_{m},Q_{m+1}}$, then similar with the above argument, applying the estimate
\eqref{3.36} of $\phi_{m+1}$ instead of \eqref{3.5}, we have the following estimate:

\begin{equation}\label{4.26}
\begin{array}{ll}
&\ds J(u_{\textbf{Q},Q_{m+1}})\vspace{0.2cm}\\
&=\ds J(u_{\bar{\textbf{Q}}_{m}})+I(z)+\frac{1}{2}\int\epsilon \tilde{V}(x)|z_{Q_{m+1}}|^{2}+\frac{1}{2}\int\epsilon^{2} |\tilde{A}(x)|^{2}|z_{Q_{m+1}}|^{2}-O\Big(\sum_{j=1}^{m}w(|Q_{m+1}-\bar{Q}_{j}|)\Big)\vspace{0.2cm}\\
&\,\,\,\,\,\,\ds+ O\Bigl[\epsilon \int \tilde{V}w_{\textbf{Q}_{m}}w_{Q_{m+1}}+ \epsilon e^{-\beta\mu} \int
\sum_{j=1}^{m}e^{-\gamma|x-Q_{j}|}\tilde{V}w_{Q_{m+1}}\vspace{0.2cm}\\
&\,\,\,\,\,\,\ds+ \epsilon \int |div\tilde{A}(x)|w_{\textbf{Q}_{m}}w_{Q_{m+1}} + \epsilon \int |\tilde{A}(x)||\nabla w_{\textbf{Q}_{m}}|w_{Q_{m+1}}+ \epsilon^{2} \int
|\tilde{A}(x)|^{2}w_{\textbf{Q}_{m}}w_{Q_{m+1}}\vspace{0.2cm}\\
&\,\,\,\,\,\,\ds+ \epsilon e^{-\beta\mu}\int |div\tilde{A}(x)|\sum_{j=1}^{m}e^{-\gamma|x-Q_{j}|}w_{Q_{m+1}} + \epsilon e^{-\beta\mu}\int
|\tilde{A}(x)|\sum_{j=1}^{m}e^{-\gamma|x-Q_{j}|}w_{Q_{m+1}}\vspace{0.2cm}\\
&\,\,\,\,\,\,\ds+ \epsilon^{2} e^{-\beta\mu}\int |\tilde{A}(x)|^{2}\sum_{j=1}^{m}e^{-\gamma|x-Q_{j}|}w_{Q_{m+1}}+\epsilon^{4}\int|\tilde{A}(x)|^{4}|w_{Q_{m+1}}|^{2}\vspace{0.2cm}\\
&\,\,\,\,\,\,\ds+\epsilon^{2}\int|\tilde{V}(x)|^{2}|w_{Q_{m+1}}|^{2}+\epsilon^{2}\int|\tilde{A}(x)|^{2}|\nabla
w_{Q_{m+1}}|^{2}+\epsilon^{2}\int|div\tilde{A}(x)|^{2}|z_{Q_{m+1}}|^{2}\Bigl)\vspace{0.2cm}\\
&\,\,\,\,\,\,\ds+O\Bigl(\epsilon^{4}\Bigl(\sum_{j=1}^{m+1}\Bigl(\int_{B_{\frac{\mu}{2}}(Q_{j})}|\tilde{A}(x)|^{4}|w_{Q_{m+1}}|^{2}\Bigl)^{\frac{1}{2}}\Bigl)^{2}+\epsilon^{2}\Bigl(
\sum_{j=1}^{m+1}\Bigl(\int_{B_{\frac{\mu}{2}}(Q_{j})}|\tilde{A}(x)|^{2}|\nabla
w_{Q_{m+1}}|^{2}\Bigl)^{\frac{1}{2}}\Bigl)^{2}\vspace{0.2cm}\\
&\,\,\,\,\,\,\ds+\epsilon^{2}\Bigl(\sum_{j=1}^{m+1}\Bigl(\int_{B_{\frac{\mu}{2}}(Q_{j})}|div\tilde{A}(x)|^{2}|w_{Q_{m+1}}|^{2}\Bigl)^{\frac{1}{2}}
\Bigl)^{2}+\epsilon^{2}\Bigl(\sum_{j=1}^{m+1}\Bigl(\int_{B_{\frac{\mu}{2}}(Q_{j})}|\tilde{V}(x)|^{2}|w_{Q_{m+1}}|^{2}\Bigl)^{\frac{1}{2}}\Bigl)^{2}\Bigl].
\end{array}
\end{equation}

By the asymptotic behavior of $V$, $A$ and $\nabla A$ at infinity, for some $\alpha < 1$, we choose $\gamma >\alpha$, then we can choose $Q_{m+1}$ such that
\begin{equation}\label{4.27}
|Q_{m+1}|\gg\frac{\max_{j=1}^{m}|\bar{Q}|_{j}+\ln\epsilon}{\gamma-\alpha},
\end{equation}
then we can get
\begin{equation}\label{4.28}
\begin{array}{ll}
&\ds\frac{1}{2}\int\epsilon \tilde{V}(x)|w_{Q_{m+1}}|^{2}+\frac{1}{2}\int\epsilon^{2} |\tilde{A}(x)|^{2}|w_{Q_{m+1}}|^{2}-O\Big(\sum_{j=1}^{m}w(|Q_{m+1}-\bar{Q}_{j}|)\Big)\vspace{0.2cm}\\
&+O\Bigl(\epsilon \ds\int \tilde{V}w_{\textbf{Q}_{m}}w_{Q_{m+1}}+ \epsilon e^{-\beta\mu} \int \sum_{j=1}^{m}e^{-\gamma|x-Q_{j}|}\tilde{V}w_{Q_{m+1}}+
\epsilon^{2} e^{-\beta\mu}\int
|\tilde{A}(x)|^{2}\sum_{j=1}^{m}e^{-\gamma|x-Q_{j}|}w_{Q_{m+1}}\vspace{0.2cm}\\
&+ \epsilon \ds\int |div\tilde{A}(x)|w_{\textbf{Q}_{m}}w_{Q_{m+1}} + \epsilon \int |\tilde{A}(x)||\nabla w_{\textbf{Q}_{m}}|w_{Q_{m+1}}+ \epsilon^{2} \int
|\tilde{A}(x)|^{2}w_{\textbf{Q}_{m}}w_{Q_{m+1}}\vspace{0.2cm}\\
&+ \epsilon e^{-\beta\mu}\ds\int |\tilde{A}(x)|\sum_{j=1}^{m}e^{-\gamma|x-Q_{j}|}w_{Q_{m+1}} + \epsilon e^{-\beta\mu}\int
|\tilde{A}(x)|\ds\sum_{j=1}^{m}e^{-\gamma|x-Q_{j}|}w_{Q_{m+1}}\vspace{0.2cm}\\
&+ \epsilon e^{-\beta\mu}\ds\int |\tilde{A}(x)|\sum_{j=1}^{m}e^{-\gamma|x-Q_{j}|}|\nabla w_{Q_{m+1}}|+\epsilon^{4}\int|\tilde{A}(x)|^{4}|w_{Q_{m+1}}|^{2}\vspace{0.2cm}\\
\end{array}
\end{equation}
\begin{equation*}
\begin{array}{ll}
&+\epsilon^{2}\ds\int|\tilde{V}(x)|^{2}|w_{Q_{m+1}}|^{2}+\epsilon^{2}\int|\tilde{A}(x)|^{2}|\nabla
w_{Q_{m+1}}|^{2}+\epsilon^{2}\ds\int|div\tilde{A}(x)|^{2}|w_{Q_{m+1}}|^{2}\Bigl)\vspace{0.2cm}\\
&+O\Bigl[\epsilon^{4}\Bigl(\ds\sum_{j=1}^{m+1}\Bigl(\ds\int_{B_{\frac{\mu}{2}}(Q_{j})}|\tilde{A}(x)|^{4}|w_{Q_{m+1}}|^{2}\Bigl)^{\frac{1}{2}}\Bigl)^{2}+\epsilon^{2}\Bigl(
\sum_{j=1}^{m+1}\Bigl(\int_{B_{\frac{\mu}{2}}(Q_{j})}|\tilde{A}(x)|^{2}|\nabla
w_{Q_{m+1}}|^{2}\Bigl)^{\frac{1}{2}}\Bigl)^{2}\vspace{0.2cm}\\
&+\epsilon^{2}\Bigl(\ds\sum_{j=1}^{m+1}\Bigl(\ds\int_{B_{\frac{\mu}{2}}(Q_{j})}|div\tilde{A}(x)|^{2}|w_{Q_{m+1}}|^{2}\Bigl)^{\frac{1}{2}}
\Bigl)^{2}+\epsilon^{2}\Bigl(\ds\sum_{j=1}^{m+1}\Bigl(\ds\int_{B_{\frac{\mu}{2}}(Q_{j})}|\tilde{V}(x)|^{2}|w_{Q_{m+1}}|^{2}\Bigl)^{\frac{1}{2}}\Bigl)^{2}\Bigl]\vspace{0.2cm}\\
&\geq\ds C\epsilon e^{-\alpha|Q_{m+1}|}-O\Big(\sum_{j=1}^{m}e^{-\eta|\bar{Q}_{j}-Q_{m+1}|}\Big)>0.
\end{array}
\end{equation*}
So
\begin{equation}\label{4.29}
\mathcal {C}_{m+1}\geq J(u_{\bar{\textbf{Q}}_{m},Q_{m+1}})>\mathcal {C}_{m}+I(z).
\end{equation}

It follows from \eqref{4.25} and \eqref{4.29} that
\begin{equation}\label{4.30}
\mathcal {C}_{m}+I(z)< \mathcal {C}_{m+1}\leq\mathcal {C}_{m}+I(z),
\end{equation}
which is impossible. Hence we prove that $\mathcal {C}_{m+1}$ can be attained at finite points in $\Omega_{m+1}$.

\end{proof}

Now we are in position to prove our main result.
\begin{proof}[\textbf{Proof of Theorem 1.2.}]
In order to prove our main result, we only need to prove that
the maximization problem
\begin{equation}\label{4.31}
\ds \max_{\textbf{Q}_{m}\in\bar{\Omega}_{m}}\mathcal {M}(\textbf{Q}_{m})
\end{equation}
has a solution $\textbf{Q}_{m}\in\Omega^{o}_{m}$, i.e., the interior of $\Omega_{m}$.

We prove it by an indirect method. Assume that $\bar{\textbf{Q}}_{m}=(\bar{Q}_{1},\ldots,\bar{Q}_{m})\in\partial\Omega_{m}$. Then there exists $(j,k)$ such
that $|\bar{Q}_{j}-\bar{Q}_{k}|=\mu$. Without loss of generality, we assume $(j, k) = (j, m)$. Then following the estimates \eqref{4.13}, \eqref{4.14}, \eqref{4.15} and
\eqref{4.21}, we have
\begin{equation}\label{4.32}
\begin{array}{ll}
&\ds \mathcal {C}_{m}=J(u_{\bar{\textbf{Q}}_{m}}) \vspace{0.2cm}\\
&\leq\ds \mathcal {C}_{m-1}+I(z)+\frac{\epsilon }{2}\int \tilde{V}(x)|w_{Q_{m+1}}|^{2}+\frac{1}{2}\int\epsilon^{2} |\tilde{A}(x)|^{2}|w_{Q_{m+1}}|^{2}\vspace{0.2cm}\\
&\,\,\,\,\,\,\ds-\frac{\vartheta}{4}\sum_{j=1}^{m-1}w(|\bar{Q}_{m} -\bar{Q}_{j} |)+O\Big(e^{-\beta\mu}\sum_{j=1}^{m-1}e^{-|\bar{Q}_{m}-\bar{Q}_{j}|}\Big)+O(\epsilon)\vspace{0.2cm}\\
&\leq\ds \mathcal
{C}_{m-1}+I(z)-\frac{\vartheta}{4}\sum_{j=1}^{m-1}w(|\bar{Q}_{m}-\bar{Q}_{j}|)+O\Big(e^{-\beta\mu}\sum_{j=1}^{m-1}e^{-|\overline{Q}_{m}-\bar{Q}_{j}|}\Big)+O(\epsilon).
\end{array}
\end{equation}

By the definition of the configuration set, we observe that given a ball of size $\mu$, there are at most $C_{N} := 6^{N}$ number of non-overlapping ball of
size $\mu$ surrounding this ball. Since $|\bar{Q}_{j}-\bar{Q}_{k}|=\mu$, we have
$$\sum_{j=1}^{m-1}w(|\bar{Q}_{m}-\bar{Q}_{j}|)= w(|\bar{Q}_{m} -\bar{Q}_{j} |) + \sum_{k\neq j}w(|\bar{Q}_{m} -\bar{Q}_{k} |) $$
and
\begin{equation}\label{4.33}
\begin{array}{ll}
\ds \sum_{k\neq j}w(|\bar{Q}_{m} -\bar{Q}_{k} |)&\leq Ce^{-\mu} + C_{N}e^{-\mu-\frac{\mu}{2}} + \ldots + C_{N}^{k}e^{-\mu-\frac{k\mu}{2}} \vspace{0.2cm}\\
&\leq\ds Ce^{-\mu}\sum_{j=1}^{\infty}e^{j(\ln C_{N}-\frac{\mu}{2})}
\leq\ds Ce^{-\mu},
\end{array}
\end{equation}
if $C_{N}<e^{\frac{\mu}{2}}$ , which is true for $\mu$ large enough.

Hence, we have
\begin{equation}\label{4.34}
C_{m} \leq C_{m-1} + I(z) + C\epsilon-\frac{\vartheta}{4}w(\mu) + O(e^{-(1+\beta)\mu}) < C_{m-1} + I(z),
\end{equation}
which contradicts to \eqref{4.5} in Lemma \ref{lem4.1}.
\end{proof}

\appendix
\section{{Some technical estimates}}\label{sa}
In this section, we give some technical estimates which are used before.

Denote
$$
\Lambda_{j}:=\Big\{x\Big||x-Q_{j}|\leq\frac{\mu}{2}\Big\},\,\,\,\, \Lambda=\bigcup_{j=1}^{m}\Lambda_{j}
\,\,\,
\text{and}\,\,\,
\Lambda^{C}=\R^{N}\backslash \Lambda.
$$

\begin{lem}\label{lemw}
For any $x\in \Lambda_{j}(j=1,...,m)$ and $\gamma\in(0, 1)$, we have
 \begin{equation}\label{e1}
\sum_{k=1}^{m} e^{-\gamma|x-Q_{k}|}
\leq e^{-\gamma|x-Q_{j}| }+ Ce^{-\frac{\gamma\mu}{2}}.
 \end{equation}
 For any $x\in \Lambda^c$, we have
 \begin{equation}\label{e2}
\sum_{k=1}^{m} e^{-\gamma|x-Q_{k}|}
\leq Ce^{-\frac{\gamma\mu}{2}}.
 \end{equation}

 \end{lem}

 \begin{proof}
 Note that given a ball of size
$\mu$, there are at most $C_{N}:=6^{N}$ number of non-overlapping ball of size $\mu$
surrounding this ball. Since $|x-Q_{j}|\leq\frac{\mu}{2},$ we have
$$
|x-Q_{k}|\geq|Q_{k}-Q_{j}|-|x-Q_{j}|\geq\frac{\mu}{2}\,\,\,\,\text{for~all}\,\,\,k\neq j.
$$
Then we have
\begin{eqnarray*}
\sum_{k=1}^{m} e^{-\gamma|x-Q_{k}|}
&=& e^{-\gamma|x-Q_{j}|}+\sum_{k\neq j}e^{-\gamma|x-Q_{k}|}
\leq  e^{-\gamma|x-Q_{j}|}+\sum_{k=1}^{\infty}C_{N}^{k}e^{-\frac{k\gamma \mu}{2}}\\
&\leq&e^{-\gamma|x-Q_{j}|}+\sum_{k=1}^{\infty}e^{k(\ln C_{N}-\frac{\gamma \mu}{2})}
\leq e^{-\gamma|x-Q_{j}|}+O(e^{-(\frac{\gamma\mu}{2}-\ln C_{N})})\\
&\leq& e^{-\gamma|x-Q_{j}|}+Ce^{-\frac{\gamma\mu}{2}},
\end{eqnarray*}
if $\ln C_{N}<\frac{\gamma\mu}{2}$, which is true for $\mu$ large enough.

The proof of \eqref{e2} is similar.
\end{proof}

\begin{proof}[\textbf{Proof of \eqref{4.9}}]
By direct computation, we have
\begin{equation*}
\begin{array}{ll}
\Bigl|\ds
Re\int\epsilon^{2}|\tilde{A}(x)|^{2}z_{Q}\bar{\varphi}_{\sigma,Q}\Bigl|&\leq\ds\epsilon^{2}\Bigl(\int|\tilde{A}(x)|^{2}|z_{Q}|^{2}\Bigl)^{\frac{1}{2}}\Bigl(\int|\tilde{A}(x)|^{2}|\bar{\varphi}_{\sigma,Q}|^{2}\Bigl)^{\frac{1}{2}}   \vspace{0.2cm}\\
&\leq\ds\delta\epsilon^{2}\int|\tilde{A}(x)|^{2}w_{Q}^{2}+C_{\delta}\epsilon^{2}\int|\tilde{A}(x)|^{2}|\varphi_{\sigma,Q}|^{2}\vspace{0.2cm}\\
&\leq\ds \delta\epsilon^{2}\int|\tilde{A}(x)|^{2}w_{Q}^{2}+C\|\varphi_{\sigma,Q}\|^{2}_{H^{1}},
\end{array}
\end{equation*}

\begin{equation*}
\begin{array}{ll}
\Bigl|\ds
Re\int\epsilon\tilde{V}(x)z_{Q}\bar{\varphi}_{\sigma,Q}\Bigl|&\leq\ds\epsilon\Bigl(\int\tilde{V}(x)|z_{Q}|^{2}\Bigl)^{\frac{1}{2}}\Bigl(\int\tilde{V}(x)|\bar{\varphi}_{\sigma,Q}|^{2}\Bigl)^{\frac{1}{2}}   \vspace{0.2cm}\\
&\leq\ds\delta\epsilon\int\tilde{V}(x)w_{Q}^{2}+C_{\delta}\epsilon\int\tilde{V}(x)|\varphi_{\sigma,Q}|^{2}\vspace{0.2cm}\\
&\leq\ds \delta\epsilon\int\tilde{V}(x)w_{Q}^{2}+C\|\varphi_{\sigma,Q}\|^{2}_{H^{1}}
\end{array}
\end{equation*}
and
\begin{equation*}
\begin{array}{ll}
&\Bigl|\ds
Re\ds\int \epsilon \tilde{A}(x) \Bigl(\frac{\nabla \varphi_{\sigma,Q}}{i}\bar{\varphi}_{\sigma,Q}-A_{0}\varphi_{\sigma,Q}\bar{\varphi}_{\sigma,Q}\Bigl) \Bigl|\vspace{0.2cm}\\
&\leq\ds\epsilon\Bigl(\int|\nabla \varphi_{\sigma,Q}|^{2}\Bigl)^{\frac{1}{2}}\Bigl(\int|\tilde{A}|^{2}|\bar{\varphi}_{\sigma,Q}|^{2}\Bigl)^{\frac{1}{2}}
\ds+C\epsilon\int|\bar{\varphi}_{\sigma,Q}|^{2}\leq\ds C\|\varphi_{\sigma,Q}\|^{2}_{H^{1}}.
\end{array}
\end{equation*}

Similarly, we can prove
\begin{equation*}
\begin{array}{ll}
\Bigl|\ds
Re\int \epsilon \tilde{A}(x) \Bigl(\frac{\nabla \varphi_{\sigma,Q}}{i}\bar{z}_{Q}-A_{0}\varphi_{\sigma,Q}\bar{z}_{Q}\Bigl) \Bigl|
&\leq\ds \delta\epsilon\int|\tilde{A}|^{2}w_{Q}^{2}+C\|\varphi_{\sigma,Q}\|^{2}_{H^{1}}.
\end{array}
\end{equation*}
Moreover, we have
\begin{equation*}
\begin{array}{ll}
&\Bigl|\ds
Re\int \frac{\epsilon}{i} \tilde{A}(x)\cdot\nabla w_{Q}\xi_{Q}\bar{\varphi}_{\sigma,Q} \Bigl| \vspace{0.2cm}\\
&= \ds\Bigl|\ds
-Re\int \frac{\epsilon}{i}w_{Q}\Bigl(div\tilde{A}\xi_{Q}\bar{\varphi}_{\sigma,Q}+i\tilde{A}\cdot A_{0}\xi_{Q}\bar{\varphi}_{\sigma,Q}+ \xi_{Q}\tilde{A}\cdot\nabla\bar{\varphi}_{\sigma,Q}  \Bigl)\Bigl| \vspace{0.2cm}\\
&\leq \ds\epsilon\int|div\tilde{A}|w_{Q}|\varphi_{\sigma,Q}|+\epsilon\int|\tilde{A}||A_{0}|w_{Q}|\varphi_{\sigma,Q}|+\epsilon\int|\tilde{A}|w_{Q}|\nabla\varphi_{\sigma,Q}|\vspace{0.2cm}\\
&\leq \ds\epsilon\Bigl(\int|div\tilde{A}|^{2}w_{Q}^{2}\Bigl)^{\frac{1}{2}}\Bigl(\int \varphi_{\sigma,Q}^{2}\Bigl)^{\frac{1}{2}}
+\epsilon\Bigl(\int|\tilde{A}|^{2}w_{Q}^{2}\Bigl)^{\frac{1}{2}}\Bigl(\int|A_{0}|^{2}\varphi_{\sigma,Q}^{2}\Bigl)^{\frac{1}{2}}
\vspace{0.2cm}\\
&\,\,\,\,\,\,\ds+\epsilon\Bigl(\int|\tilde{A}|^{2}w_{Q}^{2}\Bigl)^{\frac{1}{2}}\Bigl(\int |\nabla\varphi_{\sigma,Q}|^{2}\Bigl)^{\frac{1}{2}}\vspace{0.2cm}\\
&\leq \ds\delta\epsilon\int|div\tilde{A}|^{2}w_{Q}^{2}+C_{\delta}\epsilon\int \varphi_{\sigma,Q}^{2}
+\delta\epsilon\int|\tilde{A}|^{2}w_{Q}^{2}+C_{\delta}\epsilon\int\varphi_{\sigma,Q}^{2}
\vspace{0.2cm}\\
&\,\,\,\,\,\,\ds+\delta\epsilon\int|\tilde{A}|^{2}w_{Q}^{2}+C_{\delta}\epsilon\int |\nabla\varphi_{\sigma,Q}|^{2}\vspace{0.2cm}\\
&\leq \ds\delta\epsilon\int|div\tilde{A}|^{2}w_{Q}^{2}+\delta\epsilon\int|\tilde{A}|^{2}w_{Q}^{2}+C\|\varphi_{\sigma,Q}\|^{2}_{H^{1}},
\end{array}
\end{equation*}
where we can choose ~$\delta>0$~ small enough.
From all the estimates above, then \eqref{4.9} holds.
\end{proof}

\end{document}